\documentclass{amsart}
\usepackage{color}
\usepackage{graphicx}
\usepackage[latin2]{inputenc}
\usepackage{amsmath,amssymb,amsbsy,amsthm,amsfonts}
\usepackage{epsfig} % postscript figures
\begin{document}
%\numberwithin{equation}{section} \marginparwidth=2cm
\def\note#1{\marginpar{\small #1}}

\def\tens#1{\pmb{\mathsf{#1}}}
\def\vec#1{\boldsymbol{#1}}

\def\norm#1{\left|\!\left| #1 \right|\!\right|}
\def\fnorm#1{|\!| #1 |\!|}
\def\abs#1{\left| #1 \right|}
\def\ti{\text{I}}
\def\tii{\text{I\!I}}
\def\tiii{\text{I\!I\!I}}

\def\diver{\mathop{\mathrm{div}}\nolimits}
\def\prob{\mathop{\mathrm{Prob}}\nolimits}
\def\supp{\mathop{\mathrm{supp}}\nolimits}
\def\sgn{\mathop{\mathrm{sgn}}\nolimits}
\def\curl{\mathop{\mathrm{curl}}\nolimits}
\def\grad{\mathop{\mathrm{grad}}\nolimits}
\def\Div{\mathop{\mathrm{Div}}\nolimits}
\def\Grad{\mathop{\mathrm{Grad}}\nolimits}

\def\tr{\mathop{\mathrm{tr}}\nolimits}
\def\cof{\mathop{\mathrm{cof}}\nolimits}
\def\det{\mathop{\mathrm{det}}\nolimits}

\def\lin{\mathop{\mathrm{span}}\nolimits}
\def\pr{\noindent \textbf{Proof: }}
\def\pp#1#2{\frac{\partial #1}{\partial #2}}
\def\dd#1#2{\frac{\d #1}{\d #2}}

\newcommand{\dx}{\,\mathrm{d}x}
\newcommand{\dy}{\,\mathrm{d}y}
\newcommand{\dz}{\,\mathrm{d}z}
\newcommand{\ds}{\,\mathrm{d}s}
\newcommand{\dS}{\,\mathrm{d}S}
\newcommand{\dxi}{\,\mathrm{d}\xi}
\newcommand{\dt}{\,\mathrm{d}t}

\def\T{\mathcal{T}}
\def\R{\mathcal{R}}
\def\bx{\vec{x}}
\def\be{\vec{e}}
\def\bef{\vec{f}}
\def\bec{\vec{c}}
\def\bs{\vec{s}}
\def\ba{\vec{a}}
\def\bn{\vec{n}}
\def\bphi{\vec{\varphi}}
\def\btau{\vec{\tau}}
\def\bc{\vec{c}}
\def\bg{\vec{g}}
\def\bb{\vec{b}}

\def\bW{\tens{W}}
\def\bG{\tens{G}}
\def\bT{\tens{T}}
\def\bD{\tens{D}}
\def\bF{\tens{F}}
\def\bg{\tens{g}}
\def\bL{\tens{L}}
\def\bB{\tens{B}}
\def\bV{\tens{V}}
\def\bQ{\tens{Q}}
\def\bS{\tens{S}}
\def\bI{\tens{I}}
\def\bA{\tens{A}}
\def\bi{\vec{i}}
\def\bv{\vec{v}}
\def\bfi{\vec{\varphi}}
\def\bk{\vec{k}}
\def\b0{\vec{0}}
\def\bom{\vec{\omega}}
\def\bw{\vec{w}}
\def\p{\pi}
\def\bu{\vec{u}}

\def\ID{\mathcal{I}_{\bD}}
\def\IP{\mathcal{I}_{p}}
\def\Pn{(\mathcal{P})}
\def\Pe{(\mathcal{P}^{\eta})}
\def\Pee{(\mathcal{P}^{\varepsilon, \eta})}

\def\R{\mathbb{R}}
\def\Rd{\mathbb{R}^d}
\def\Rdp{\mathbb{R}^{d+1}_+}
\def\Rp{\mathbb{R}_+}

\def\Ln#1{L^{#1}_{\bn}}

\def\Wn#1{W^{1,#1}_{\bn}}

\def\Lnd#1{L^{#1}_{\bn, \diver}}

\def\Wnd#1{W^{1,#1}_{\bn, \diver}}

\def\Wndm#1{W^{-1,#1}_{\bn, \diver}}

\def\Wnm#1{W^{-1,#1}_{\bn}}

\def\Lb#1{L^{#1}(\partial \Omega)}

\def\Lnt#1{L^{#1}_{\bn, \btau}}

\def\Wnt#1{W^{1,#1}_{\bn, \btau}}

\def\Lnd#1{L^{#1}_{\bn, \btau, \diver}}

\def\Wntd#1{W^{1,#1}_{\bn, \btau, \diver}}

\def\Wntdm#1{W^{-1,#1}_{\bn,\btau, \diver}}

\def\Wntm#1{W^{-1,#1}_{\bn, \btau}}

\def\W0dm#1{W^{-1,#1}_{0, \diver}}

\def\W0d#1{W^{1,#1}_{0, \diver}}

\def\W0m#1{W^{-1,#1}_{0}}
\def\essup{\mathop{\mathrm{ess \, sup}}\limits}

\def\W0#1{W^{1,#1}_{0}}

\def\mA{\mathcal{A}}

\def\bbbone{{\mathchoice {\rm 1\mskip-4mu l}
{\rm 1\mskip-4mu l} {\rm 1\mskip-4.5mu l} {\rm 1\mskip-5mu l}}}

%------------------------------------------------
\newtheorem{Theorem}{Theorem}[section]
%{\theorembodyfont{\rmfamily} \newtheorem{Example}{Example}}
\newtheorem{Example}{Example}[section]
\newtheorem{Lemma}{Lemma}[section]
\newtheorem{Proposition}{Proposition}[section]
\newtheorem{Rem}{Remark}[section]
\newtheorem{Def}{Definition}[section]
\newtheorem{Col}{Corollary}[section]
\numberwithin{equation}{section}
%\tableofcontents

%\title{On scalar hyperbolic conservation laws with a discontinuous flux\tnoteref{t1}}

%\tnotetext[t1]{M. Bul\'{\i}\v{c}ek thanks to Jind\v{r}ich Ne\v{c}as Center for Mathematical
%Modeling, the project LC06052 financed by M\v SMT. P. Gwiazda and A. \'{S}wierczewska-Gwiazda acknowledge the Grant of Ministry of
%Science and Higher Education, Nr N201 033 32/2269. J. M\'{a}lek's contribution is a part of the research
%project MSM 0021620839 financed by M\v{S}MT; the support of GA\v{C}R
%201/08/0315 is also acknowledged.}
%
%\author[i1]{Miroslav Bul\'{\i}\v{c}ek}
%\ead{mbul8060@karlin.mff.cuni.cz}
%
%\author[i2]{Piotr Gwiazda}
%\ead{P.Gwiazda@mimuw.edu.pl}
%
%\author[i1]{Josef M\'alek\corref{c1}}
%\ead{malek@karlin.mff.cuni.cz}
%
%\author[i2]{Agnieszka \'{S}wierczewska-Gwiazda}
%\ead{aswiercz@mimuw.edu.pl}
%
%\cortext[c1]{Corresponding author}
%
%\address[i1]{Mathematical Institute of Charles University,
%Sokolovsk\'{a} 83, 186 75 Prague, Czech Republic}
%
%\address[i2]{Institute of Applied Mathematics, University of Warsaw,
%ul. Banacha 2, 02-097 Warsaw, Poland}
%
%
%
%
%\begin{keyword}
%hyperbolic scalar conservation laws\sep  discontinuous flux \sep entropy solution
%\MSC 35L65\sep 35R05
%\end{keyword}
%\maketitle
% \subjclass[2000]{35L65,35R05}

\title[Conservation laws with discontinuous flux]{On unified theory for scalar conservation laws with fluxes and sources  discontinuous with respect to the unknown}\thanks{M.~Bul\'{\i}\v{c}ek thanks the project MORE (ERC-CZ project LL1202 financed by the Ministry of Education, Youth and Sports, Czech Republic). M.~Bul\'{\i}\v{c}ek is a member of the Ne\v{c}as Center for Mathematical Modeling. A.~\'{S}wierczewska-Gwiazda acknowledges the support of the National  Science Centre,  DEC-2012/05/E/ST1/02218.  The research was partially supported by the Warsaw Center of Mathematics and Computer Science.  P.~Gwiazda received support from the National Science Centre (Poland), 2015/18/M/ST1/00075.}

\author[M. Bul\'{\i}\v{c}ek]{Miroslav Bul\'i\v{c}ek}
\address{Miroslav Bul\'i\v{c}ek, Mathematical Institute, Faculty of Mathematics and Physics, Charles University\\ Sokolovsk\'{a} 83,
186 75 Praha 8, Czech Republic}
\email{mbul8060@karlin.mff.cuni.cz}

\author[P. Gwiazda]{Piotr Gwiazda}
\address{Piotr Gwiazda, Institute of Applied Mathematics, University of Warsaw, ul. Banacha 2, 02-097 Warsaw, Poland}
\address{Institute of Mathematics of Polish Academy of Sciences, \'Sniadeckich 8,
00-656 Warszawa Poland}
\email{P.Gwiazda@mimuw.edu.pl}

\author[A. \'Swierczewska-Gwiazda]{Agnieszka \'Swierczewska-Gwiazda}
\address{Agnieszka \'Swierczewska-Gwiazda, Institute of Applied Mathematics, University of Warsaw, ul. Banacha 2, 02-097 Warsaw, Poland}
\email{aswiercz@mimuw.edu.pl}

\begin{abstract}

We deal with the Cauchy problem for multi-dimensional scalar conservation laws, where the fluxes and the source terms can be discontinuous functions of the unknown.
The main novelty of the paper is the  introduction of  a~kinetic formulation for the considered problem. To handle the discontinuities we work in the framework of re-parametrization of the flux and the source functions, which was previously used for  Kru\v{z}kov entropy solutions. Within this approach we obtain a fairly complete picture: existence of entropy measure valued solutions, entropy weak solutions and their equivalence to the kinetic solution. The results of existence and uniqueness follow under the assumption of H\"{o}lder continuity at zero of the flux. The source term, what is another novelty for the  studies on problems with discontinuous flux, is only assumed to be  one-side Lipschitz, not necessarily monotone function.

\end{abstract}



\maketitle

\section{Introduction} \label{S1}
We focus on the Cauchy problem for a scalar hyperbolic balance law of the following form
% Such a problem is described by the following equations
\begin{align}
\left.\begin{aligned}
\partial_t u + \diver \bF(u)&=G(u) &&\textrm{in }\mathbb{R}_+^{d+1},\\
u(0)&=u_0 &&\textrm{in }\mathbb{R}^{d},
\end{aligned}\right.
\label{E1}
\end{align}
where $\mathbb{R}^{d+1}_+:=(0,\infty)\times \Rd$, $d\ge 1$ denotes an arbitrary spatial dimension, $u:\mathbb{R}^{d+1}_+ \to \mathbb{R}$  is an unknown function,  $\bF: \mathbb{R} \to \Rd$ is a given flux of the quantity $u$, $u_0:\mathbb{R}^d\to \mathbb{R}$ is the initial condition  and $G: \mathbb{R} \to \mathbb{R}$ is the given source term. In addition we assume that $u$  vanishes as $|x|\to \infty$. The main goal of the paper is to build a sufficiently robust framework that is capable to cover as general class of fluxes and source terms as possible. In particular, we want to focus on the cases when both quantities are discontinuous functions of the unknown $u$. The starting point is  the method introduced in \cite{BuGwMaSw2011}, where the authors considered the problem without the source term on the right-hand side and roughly speaking with the flux having jump discontinuities with respect to $u$ and showed the existence and uniqueness result of a weak entropy solution. Later the framework  was extended to fluxes discontinuous both in $x$ and $u$ in \cite{BuGwGw13} and further in~\cite{GwSwWiZi2014}, where the formulation encounters also the source term on the right-hand side, but  under some rather  restrictive assumptions like continuity and monotonicity. The problem with a one-side Lipschitz source term (but with a continuous flux) was studied in~\cite{GwSw2005} by the methods of set-valued analysis. However, in none of these works the kinetic formulation for such problems has been introduced. The existence proof of entropy weak solutions followed the scheme of regularising  the discontinuous flux  and adapting in some way the ideas of Kru\v{z}kov~\cite{Kr70} and the framework of entropy measure valued solutions (DiPerna~\cite{DiPerna}, Szepessy~\cite{Sz89a}). In the current paper we significantly relax the assumptions on the source term in comparison to~\cite{GwSwWiZi2014} and  we define the \emph{kinetic formulation} and  \emph{the weak entropy solution} to \eqref{E1} and show their equivalence. Finally, we present the constructive proof of existence of a kinetic solution. It means that instead of a standard approximation by a diffusion term we approximate the kinetic equation by an equation of the Boltzmann type. This procedure is usually referred as the \emph{kinetic approximation}.

Furthermore, we also incorporate the notion of \emph{entropy measure valued solution}, which will be shown to be equivalent to the above ones and introduce and prove the stability of  measure valued solutions with respect to data, which will be the key tool  for the uniqueness and the independence of  solutions of  parametrization.

The techniques presented here apply to a more general setting, in particular to the fluxes and sources which in addition are $x-$discontinuous  in a way that they satisfy  the structural assumptions introduced in \cite{AudussePerthame,Panov} and further essentially generalized in \cite{BuGwGw13}, where however the discontinuous source term is not taken into account.

\subsection{The classical results for smooth $\bF$ and $G$}
Our aim is the unification of the well-posedness theory for discontinuous scalar conservation laws. To understand the interplay between entropy weak solutions and kinetic formulation we first briefly describe the development of the theory for problems involving smooth fluxes $\bF$ and sources $G$. It is well known that \eqref{E1} may not be solvable globally in time in the classical sense even if the initial data are smooth. On the other hand the class of weak solutions is too wide to provide uniqueness. Therefore, one surely needs to modify the meaning of a solution in a way that some additional  selection criteria allow to choose a unique solution.  The  celebrated idea of Kru\v{z}kov, see  \cite{Kr70}, coming from the viscous approximation, was to add an additional constraint for the solution, the so-called \emph{entropy inequality}, which takes the following form
\begin{equation}
\begin{split}
\partial_t |u-k| + \diver \left( \sgn(u-k) (\bF(u)-\bF(k)) \right) \le G(u) \sgn (u-k) \quad \textrm{in }\mathbb{R}^{d+1}_+, \label{Kruzx}
\end{split}
\end{equation}
for any $k\in\R$ or more generally (but equivalently)
\begin{equation}
\begin{split}
\partial_t E(u)+ \diver \bQ(u)\le G(u) E'(u) \quad \textrm{in }\mathbb{R}^{d+1}_+, \label{Kruz2x}
\end{split}
\end{equation}
where $E$ is an arbitrary smooth convex function (entropy) and $\bQ$ (flux) satisfies
$$
 \bQ'(u)=E'(u)\bF'(u).
$$
The system \eqref{E1} completed by \eqref{Kruzx} leads to  uniqueness and existence of a solution provided that the flux $\bF$ and the source term $G$ are globally\footnote{We refer here to the original paper of \cite{Kr70}, where the case of bounded initial data is studied, and for more general result with $L^1$ data, growth assumptions on $\bF$ and the H\"{o}lder continuity of the flux only at the point zero, see \cite{Sz89a}.} Lipschitz continuous, $G(0)=0$ and $u_0\in L^1(\mathbb{R}^d)$. Later on, it was observed in \cite{PT91}  that the essential properties of the system \eqref{E1} can be deduced from the equivalent  \emph{kinetic formulation} of  conservation laws, which has the form
\begin{align}
\left.\begin{aligned}
\partial_t \chi(\xi,u) + \bF'(\xi)\cdot \nabla_x \chi(\xi,u) + G(\xi)\partial_{\xi} \chi(\xi,u)-G(0)\delta_0(\xi)&=\partial_{\xi} m \textrm{ in }\mathbb{R}_+^{d+2}%,\\
%\chi(\cdot,u_0)&=\chi(0) &&\textrm{in }\mathbb{R}^{d+1},
\end{aligned}\right.
\label{E1kx}
\end{align}
where $m$ is a nonnegative measure, $\delta_0$ is the Dirac measure and  $\chi:\R^2\to\{1,-1,0\}$ is defined as follows
\begin{equation}\label{chi}
\chi(z,u):=
\left\{
\begin{aligned}
&1 &&\textrm{if }  0<z<u,\\
&-1 &&\textrm{if }\  u<z<0,\\
&0 &&\textrm{otherwise.}
\end{aligned}
\right.
\end{equation}
Indeed, it was shown in\footnote{Although the proof of such equivalence is given there without the source term $G$, it is a direct consequence of the method developed there. Moreover, we shall provide a rigorous proof of such equivalence in much more general setting here.} \cite{PT91} that $u$ solves \eqref{E1} completed by the entropy inequality \eqref{Kruzx} if and only if there exists a nonnegative measure $m$ such that  \eqref{E1kx} holds.
The kinetic equation is a linear transport equation for $\chi$, what simplifies a lot of considerations and even leads to several new results. Also the concept of the kinetic formulation is used for numerical schemes, see~\cite{Perthame02, Bo2002, MaPe2003}.

On the other hand, it is evident that both methods - kinetic formulation or Kru\v{z}kov entropies  - heavily rely on the smoothness of the flux and the source term, which are required to be Lipschitz or at least absolutely continuous, cf.~\cite{Kr70,Perthame02}.

As the starting point to overcome the problem of the lack of continuity, it was pointed out in \cite{BuGwMaSw2011} (see also \cite{Carrillo2003} for a similar approach) that for smooth fluxes $\bF$ one can introduce a huge class of equivalent formulations to the Kru\v{z}kov entropy formulation. Indeed, denoting for any smooth increasing function (bijection) $U:\mathbb{R}\to \mathbb{R}$ a new flux $\bA:=\bF\circ U$ and a new source term $a:=G\circ U$ it was shown (without the presence of the source term, which will be in particular  carefully treated in this paper) in~\cite{BuGwMaSw2011} that  $u$ solves \eqref{E1} and \eqref{Kruzx} if and only if  there exists $v$ solving for all $k\in \mathbb{R}$ the following problem:
\begin{align}
\partial_t U(v) + \diver \bA(v)&=a(v) &&\textrm{in } \mathbb{R}^{d+1}_+,\label{E1B}\\
U(v(0))&=u_0 &&\textrm{in }\mathbb{R}^d,\label{E12B}\\
\partial_t |U(v)-U(k)| \!+ \diver (\sgn (v-k) (\bA(v)-\bA(k)))&\le a(v)\sgn (v-k)\!\!\!\! &&\textrm{in }
\mathbb{R}^{d+1}_+.\label{KruzB}
\end{align}
Moreover, it follows from  \eqref{KruzB}, that such solutions satisfy (provided that $u$ is bounded, otherwise one needs to prescribe further growth assumptions on $E$, $a$ and $\bA$)
\begin{align}
\partial_t Q_U(v)+ \diver (Q_{\bA}(v))&\le a(v)E'(v)  &&\textrm{in }
\mathbb{R}^{d+1}_+\label{KruzC}
\end{align}
for arbitrary convex $E\in \mathcal{C}^1(\mathbb{R})$ with fluxes $Q_U$ and $Q_{\bA}$ given by
$$
Q'_U=U'E', \qquad Q'_{\bA}=\bA'E'.
$$

In analogue to equivalence of \eqref{Kruzx} and \eqref{E1kx} it can be  expected that  a kinetic formulation has  now  the following form
\begin{align}
 U'(\xi)\partial_t\chi(\xi,v) + \bA'(\xi)\cdot \nabla_x \chi(\xi,v) + a(\xi)\partial_{\xi} \chi(\xi,u)-a(0)\delta_0(\xi)&=\partial_{\xi} m \textrm{ in }\mathbb{R}_+^{d+2}%,\\
%\chi(0)&=\chi(\cdot,U^{-1}(u_0)) &&\textrm{in }\mathbb{R}^{d+1},
\label{E1kB}
\end{align}
with $m$ being again a nonnegative measure,   possibly different to the one appearing in \eqref{E1kx}. The fact that \eqref{E1kB} is indeed   equivalent reformulation of \eqref{KruzB} is one of the key results of the present paper, see Theorem~\ref{TT1}.
Although it is not at all obvious at the moment, but apparently both formulations \eqref{KruzB} and  \eqref{E1kB} will require only continuity of the corresponding fluxes, i.e., $\bA$, $U$ and $a$. Thus an immediate conclusion is that one could treat the case of discontinuous fluxes and  the source term unless it is possible to find a continuous parametrization of the discontinuity. This  means that we can   find  an appropriate  $U$ such that $\bF\circ U$ and $G\circ U$ can be understood as continuous functions. Having in mind uniqueness of solutions in a certain class, we expect that some additional property of $\bF$ and $G$ might be required. In the original attempt Kru\v{z}kov assumed that both $\bF$ and $G$ are Lipschitz continuous, which is surely the property that we want to avoid here. Later, however, still for continuous fluxes, it was pointed out in \cite{Sz89a} that for flux $\bF$ one needs only the H\"{o}lder continuity at zero (which somehow corresponds to the H\"{o}lder continuity for boundary values),  and concerning the right hand side that it is Lipschitz and globally continuous, see the counterexamples for the uniqueness in \cite{Sz89a} for flux not being smooth enough at zero. The generalization of these results is that we are able to treat discontinuous quantities with the only proviso that the flux $\bF$ is H\"{o}lder continuous at zero and that $G$ is only \emph{ one-side Lipschitz}, i.e., it is a sum of a Lipschitz and a monotone function.

\subsection{Formulation of the problem with discontinuities}

Finally, we shall introduce assumptions on $\bF$ and $G$  more precisely and formulate  main results of the paper. First, it seems to be more convenient not to consider $\bF$ and $G$ as  functions of $u$ but rather ``identify" them with  proper subsets of $\mathbb{R}^{d+1}$ and $\mathbb{R}^2$ respectively, the \emph{graphs}
$$
\mathcal{A}_{\bF}\subset \mathbb{R}^d \times \mathbb{R}, \qquad \mathcal{A}_G \subset \mathbb{R}\times \mathbb{R}.
$$
The assumption that  $G$ is  one-side Lipschitz in the language of the graphs means that there exist a constant $L>0$ such that for every $(g_1,u_1), (g_2,u_2)\in\mA_{G}$ the condition $\sgn(u_1-u_2)(g_1-g_2)\le L|u_1-u_2|$ holds.

For such given graphs $\mA_{\bF}$ and $\mA_G$, we introduce the generalization of \eqref{E1} in the following way: Find $u:\mathbb{R}^{d+1}_+ \to \mathbb{R}$, $\bF:\mathbb{R}^{d+1}_+ \to \mathbb{R}^d$ and $G:\mathbb{R}^{d+1}_+ \to \mathbb{R}$ solving
\begin{align}
\left.\begin{aligned}
\partial_t u + \diver \bF &=G, \quad (\bF,u)\in \mathcal{A}_{\bF},\quad   (g,u)\in \mathcal{A}_G &&\textrm{in }\mathbb{R}_+^{d+1},\\
u(0)&=u_0 &&\textrm{in }\mathbb{R}^{d}.
\end{aligned}\right.
\label{E1g}
\end{align}
The key assumption imposed on the graphs $\mA_{\bF}$ and $\mA_G$ is that they admit the \emph{continuous (monotone) parameterizations}, i.e., we require that there exists a~Lipschitz continuous nondecreasing function $U:\mathbb{R}\to \mathbb{R}$  such that $U(\mathbb{R})=\mathbb{R}$, and continuous functions $\bA:\mathbb{R}\to \mathbb{R}^d$ and $a:\mathbb{R}\to \mathbb{R}$ such that for all $s\in \mathbb{R}$ there holds
\begin{equation}
(\bA(s),U(s))\in \mathcal{A}_{\bF},\qquad (a(s),U(s))\in \mathcal{A}_G. \label{param}
\end{equation}
It is evident that without loss of generality we may assume that $U(0)=0$. Moreover, in case  the graphs can be represented by continuous functions, we can simply set $U(s)=s$, $\bA(s)=\bF(s)$ and $a(s)=G(s)$. It is also worth of noticing here, that the assumption \eqref{param} directly leads to the certain maximality of the corresponding graphs, or in other words, in case of discontinuous $\bF$ and $G$, we fill the possible jump by the whole interval assuming then the multi-valued functions.

Under the assumption \eqref{param}, we finally define several notions of entropy solution to \eqref{E1g}. For simplicity, we restrict ourselves only to the case when the solution is assumed to be locally bounded. In case we consider that it is only locally an $L^1$ quantity, we have to impose more growth assumptions at infinity on $U$, $a$ and $\bA$, see also \eqref{AF_mon}--\eqref{Ag_mon} and Subsection~\ref{subsec1.3} for existence result with $L^1$ data. First, we focus on the definition of the entropy weak solution, which corresponds to the inequality \eqref{KruzC}. Let us for simplicity fix $T<\infty$, which can   be arbitrary.
\begin{Def}[Entropy formulation]
\label{D1}
Let $u_0 \in L^{\infty}_{loc}(\mathbb{R}^d)$. Assume that graphs $\mA_{\bF}$ and $\mA_G$ admit a continuous parametrization \eqref{param}. We say that a function $u:\mathbb{R}^{d+1}_+\to\R$
 is an entropy solution to  \eqref{E1g} and \eqref{param} with the initial data $u_0$ if for every $T<\infty$ it holds $u\in L^{\infty}(0,T; L^{\infty}_{loc}(\mathbb{R}^d))$ and:
\begin{itemize}
\item[1)]
There exists a measurable $v:\mathbb{R}^{d+1}_+\to \mathbb{R}$ and $v_0:\mathbb{R}^d\to \mathbb{R}$ such that $u(t,x)=U(v(t,x))$ almost everywhere in $\mathbb{R}^{d+1}_+$ and  $u_0(x)=U(v_0(x))$ almost everywhere in $\mathbb{R}^d$;
\item[2)] For  all smooth convex functions $E:\mathbb{R}\to \mathbb{R}$ fulfilling  $E''\in \mathcal{D}(\mathbb{R})$, and for all nonnegative $\varphi \in \mathcal{D}(-\infty,T; \mathcal{D}(\mathbb{R}^d))$ there holds
\begin{equation}
\begin{aligned}
&\int_{\mathbb{R}^{d+1}_+} Q_U(v(t,x))\partial_t \varphi(t,x) + Q_{\bA}(v(t,x))\cdot \nabla_x \varphi(t,x)\dx \dt \\
&\ge -\int_{\mathbb{R}^{d+1}_+} a(v(t,x))E'(v(t,x))\varphi(t,x)\dx \dt -\int_{\mathbb{R}}Q_U(v_0(x))\varphi(0,x)  \dx, \label{KruzDF}
\end{aligned}
\end{equation}
\end{itemize}
where  fluxes $Q_U$ and $Q_{\bA}$ are given as
\begin{equation}\label{DFflux}
Q'_U=U'E', \qquad Q'_{\bA}=\bA'E'.
\end{equation}
\end{Def}
Note here that although there appear derivatives of $\bA$ and $U$ in \eqref{DFflux}, we just need a continuity of these quantities. Indeed, $Q_U$ and $Q_{\bA}$ can be equivalently (up to an additive constant) defined as
$$
\begin{aligned}
Q_U(s)=-\int_0^s U(t)E''(t) \dt + U(s)E'(s), \quad Q_{\bA}(s)=\int_0^s \bA(t)E''(t) \dt + \bA(s)E'(s).
\end{aligned}
$$
In addition, since $u$ and $u_0$ are locally bounded, it follows from the properties of $U$ that also $v$ and $v_0$ are locally bounded and therefore all integrals in \eqref{KruzDF} are well defined.

Next, we define a notion of the kinetic solution. Recall here the definition of the indicator function $\chi$ given in \eqref{chi}. First, we start with a definition that requires smoother $U$, $a$ and $\bA$.
\begin{Def}[Kinetic formulation I]
\label{D2}
Let $u_0 \in L^\infty_{loc}(\mathbb{R}^d)$. Assume that graphs $\mA_{\bF}$ and $\mA_G$ admit a continuous parametrization \eqref{param}. Moreover, assume that $U,\bA \in W^{1,1}_{loc}(\mathbb{R})$. We say that a function $u:\mathbb{R}^{d+1}_+\to\R$
 is a kinetic solution to  \eqref{E1g} and \eqref{param} with the initial data $u_0$  if for every $T<\infty$ it holds $u\in L^{\infty}(0,T; L^{\infty}_{loc}(\mathbb{R}^d))$ and:
\begin{itemize}
\item[1)] There exists $v:\mathbb{R}^{d+1}_+\to \mathbb{R}$ and $v_0:\mathbb{R}^d\to \mathbb{R}$ such that $u(t,x)=U(v(t,x))$ almost everywhere in $\mathbb{R}^{d+1}_+$ and $u_0(x)=U(v_0(x))$ almost everywhere in $\mathbb{R}^d$
\item[2)] There exists nonnegative locally (with respect to $(t,x)$) compactly  supported (with respect to $\xi)$\footnote{ We mean that for every compact set $K\subset\R\times\R^d$ the set $\{\xi:m(t,x)\neq 0, (t,x)\in K\}$ is compact.} measure $m\in \mathcal{M}(\mathbb{R}^{d+2}_+)$ such that for all $\varphi \in \mathcal{D}(\mathbb{R}^{d+2})$ there holds
\begin{align}
\left.\begin{aligned}
&\langle m,\partial_{\xi}\varphi\rangle=\int_{\mathbb{R}^{d+1}} U'(\xi)\chi(\xi,v_0(x))\varphi(0,x,\xi) \dxi\dx\\
&+\int_{\mathbb{R}^{d+2}_+} \Big(( U'(\xi)\chi(\xi,v(t,x))\partial_t \varphi(t,x,\xi)  + \bA'(\xi)\chi(\xi,v(t,x))\cdot \nabla_x \varphi(t,x,\xi)\\
&\qquad + a'(\xi) \chi(\xi,v(t,x))\varphi(t,x,\xi) + a(\xi) \chi(\xi,v(t,x))\partial_{\xi}\varphi(t,x,\xi)\Big)  \dxi \dx\dt\\
&+\int_{\mathbb{R}^{d+1}_+} a(0)\varphi(t,x,0) \dx \dt.
\end{aligned}\right.
\label{E1kBweak}
\end{align}
\end{itemize}
\end{Def}
The appearance of derivatives of $\bA$, $a$ and $U$ suggests that the assumption of an absolute continuity is needed to provide the definition to be meaningful. Apparently the particular form of $\chi$,  in particular that the derivative of $\chi$ is a measure allows to integrate by parts in
%Note here, that in \eqref{E1kBweak} appears derivatives of $\bA$, $a$ and $U$ so it might seem that for meaningful definition one should require absolute continuity of these quantities. However, since the derivative of $\chi$ with respect to $\xi$ is a measure one can still integrate by parts in
\eqref{E1kBweak} to obtain the following ``equivalent" definition
\begin{Def}[Kinetic formulation II]
\label{D3}
Let $u_0 \in L^\infty_{loc}(\mathbb{R}^d)$. Assume that graphs $\mA_{\bF}$ and $\mA_G$ admit a continuous parametrization \eqref{param}.
We say that a function $u:\mathbb{R}^{d+1}_+\to\R$
 is a kinetic solution to  \eqref{E1g} and \eqref{param} with the initial data $u_0$  if for every $T<\infty$ it holds $u\in L^{\infty}(0,T; L^{\infty}_{loc}(\mathbb{R}^d))$ and:
\begin{itemize}
\item[1)] There exist $v:\mathbb{R}^{d+1}_+ \to \mathbb{R}$ and $v_0:\mathbb{R}^d\to \mathbb{R}$ such that $u(t,x)=U(v(t,x))$ almost everywhere in $\mathbb{R}^{d+1}_+$ and $u_0(x)=U(v_0(x))$ almost everywhere in $\mathbb{R}^d$;
\item[2)] There exists nonnegative locally (with respect to $(t,x)$) compactly  supported (with respect to $\xi)$ measure $m\in \mathcal{M}(\mathbb{R}^{d+2}_+)$ such that for all $\varphi \in \mathcal{D}(\mathbb{R}^{d+2})$ there holds
\begin{align}
\left.\begin{aligned}
\langle m,\partial_{\xi}\varphi\rangle&=-\int_{\mathbb{R}^{d+1}} U(\xi)\chi(\xi,v_0)\partial_{\xi}\varphi(0,x,\xi) \dxi\dx\\
&-\int_{\mathbb{R}^{d+2}_+} U(\xi)\chi(\xi,v(t,x))\partial_{t,\xi} \varphi(t,x,\xi)\dxi\dx \dt\\
&-\int_{\mathbb{R}^{d+1}_+}\langle U(\xi)\partial_t \varphi(t,x,\xi),\partial_{\xi}\chi(\xi,v(t,x))\rangle  \dx\dt\\
&-\int_{\mathbb{R}^{d+2}_+} \bA(\xi)\chi(\xi,v(t,x))\cdot \nabla_{x} \partial_{\xi}\varphi(t,x,\xi) \dxi\dx\dt\\
&-\int_{\mathbb{R}^{d+1}_+}\langle\bA(\xi)\cdot \nabla_x \varphi(t,x,\xi),\partial_{\xi}\chi(\xi,v(t,x))\rangle\dx\dt\\
 &+\int_{\mathbb{R}^{d+1}_+} a(0)\varphi(t,x,0)   -\langle a(\xi)\varphi(t,x,\xi), \partial_{\xi}\chi(\xi,v(t,x))\rangle\dx \dt\\
  &+\int_{\mathbb{R}^{d+2}_+} a(\xi) \chi(\xi,v(t,x))\partial_{\xi}\varphi(t,x,\xi)  \dxi \dx\dt.
\end{aligned}\right.
\label{E1kBweak2}
\end{align}
\end{itemize}
\end{Def}
The first key result of the paper is the following statement.
\begin{Theorem}[Equivalence of various formulations]\label{TT1}
Let $u_0 \in L^\infty_{loc}(\mathbb{R}^d)$. Assume that graphs $\mA_{\bF}$ and $\mA_G$ admit a continuous parametrization \eqref{param}. Then $u$ is an entropy weak solution in the sense of Definition~\ref{D1} if and only if it is a kinetic solution in the sense of  Definition~\ref{D3}. Moreover, if $U, a, \bA \in W^{1,1}_{loc}(\mathbb{R})$ then  Definitions~\ref{D2} and \ref{D3} are equivalent.
\end{Theorem}
\begin{Rem}
If in addition $U$ is a strictly increasing function then $u$ is the standard  entropy weak solution satisfying \eqref{Kruz2x}.
This equivalence has already been established in previous papers, see \cite{BuGwGw13,BuGwMaSw2011,GwSwWiZi2014}, however for the conservation law without  the source term $a$ or $G$, respectively. Anyway, the presence of a source term  can be easily handled by the same technique.
\end{Rem}
We would like to point out that this is the first result showing the  the equivalence of the three notions of solutions for more general case than the classical level of conservation laws with continuous fluxes, in particular for the case, where the kinetic formulation in the form of \eqref{E1kx} is not well defined due to the possible discontinuity of $\bF$.

\subsection{Existence, uniqueness and independence of the parametrization}\label{subsec1.3}
The second key result of the paper is the existence and uniqueness of an entropy solution. However, to obtain such a result one is forced to add several assumptions on the flux and also on the source term. Therefore,  we shall require the certain H\"{o}lder continuity of $\mA_{\bF}$ at zero,  one side Lipschitz continuity of the graph $\mA_G$ and certain one side asymptotical boundedness of $\mA_G$ in terms of $U'$, i.e., we assume that there exist $L,\alpha>0$ such that
\begin{align}
|\bA(s)-\bA(0)|&\le L(|U(s)|+|U(s)|^{\alpha})&& \textrm{for all } s\in \mathbb{R}, \label{AF_mon}\\
a(s)-a(v) &\le L(U(s)-U(v)) &&\textrm{for all } s\ge v, \label{Ag_mon}\\
a(0)&=0, \label{Ag_0}\\
\frac{a(s)}{s}&\le LU'(s)&&\textrm{for all } s\in (-\infty,-L)\cup (L,\infty), \label{Ag_bound}
\end{align}
where the last inequality is understood in the sense of distribution in case $U\notin \mathcal{C}^1$. In fact, \eqref{Ag_bound} is exactly the condition appearing later in a~priori estimates, however can be replaced be the following stronger one
\begin{equation}\label{equiv_c}
\begin{split}
\limsup_{|s|\to \infty}\frac{a(s)}{s} <\infty, \qquad \liminf_{|s|\to \infty} U'(s)>0.
\end{split}
\end{equation}
Notice here, that \eqref{equiv_c} directly implies \eqref{Ag_bound}. In addition the set of assumptions \eqref{AF_mon}--\eqref{Ag_bound} can be further relaxed. In case, we assume that the initial data are globally bounded and we want to deal only with  globally bounded solutions, then the first part of the right hand side in \eqref{AF_mon} is not needed. On the other hand, if the initial data belong only to $L^1(\mathbb{R}^d)$ (and so we cannot hope for a bounded solution), then the assumption \eqref{AF_mon} is essential, however the last assumption \eqref{Ag_bound} is irrelevant. Note that as $\mA_G$ is one side Lipschitz, then conditions \eqref{Ag_mon} and \eqref{Ag_bound} are automatically satisfied for every Lipschitz parametrization $U$.

Finally, to prove also independence of the solution of the graph representation, we introduce a notion of the \emph{equivalent} parametrization, i.e., we say that two parametrizations $(U^i,\bA^i,a^i)$ with $i=1,2$ are equivalent if and only if there exists a strictly increasing $\tilde{U}(\R)\in \mathcal{C}^1(\R)$ with a continuous inverse such that for all $s\in \R$ there holds
\begin{equation} \label{eqpa}
U^1(\tilde{U}(s))=U^2(s), \quad \bA^1(\tilde{U}(s))=\bA^2(s), \quad a^1(\tilde{U}(s))=a^2(s).
\end{equation}
The second key result of the paper is the following.
\begin{Theorem}[Existence and uniqueness]
Let $\mA_{\bF}$ and $\mA_{G}$ admit a continuous parametrization $(U,\bA,a)$ satisfying \eqref{param}. In addition, assume that the parametrization satisfies \eqref{AF_mon}--\eqref{Ag_bound} with $\alpha\ge \frac{d-1}{d}$. Then for any $u_0\in L^1(\R^d)\cap L^{\infty}(\R^d)$ there exists an entropy weak solution $u\in \mathcal{C}(0,T;L^1(\R)) \cap L^{\infty}(0,T; L^{\infty}(\R^d))$ in sense of Definition~\ref{D1}. Moreover, let $(U^1,\bA^1,a^1)$ and $(U^2,\bA^2,a^2)$ be two equivalent parametrizations, i.e., fulfill \eqref{eqpa}, and let $u^1$ and $u^2$ be two corresponding entropy weak solutions with initial data $u_0$. Then $u^1=u^2$ almost everywhere in~$\R^{d+1}_+$. In~addition, if $a^1$ satisfies
\begin{equation}
a^1(s)-a^1(v) \le L(U(s)-U(v)) -\gamma(s-v) \textrm{ for all } s\ge v,
\label{strictnm}
\end{equation}
for some nonnegative $\gamma\in \mathcal{C}(\R)$ fulfilling
$$
\gamma(t)=0 \implies t=0,
$$
then $v$ from Definition~\ref{D1}, i.e., $v$ fulfilling $U(v)=u$, where $u$ is an entropy weak solution, is uniquely determined by $v_0$, which is assumed to satisfy $U(v_0)=u_0$.
\label{TT2}
\end{Theorem}
This result gives the final answer on the question of the existence and the uniqueness of a bounded entropy weak solution for problems, where the flux and the source term can be suitably parameterized. Furthermore, this result can be still relaxed onto the case when the initial data are not bounded in the following way.
\begin{Theorem}[Existence and uniqueness]
Let $\mA_{\bF}$ and $\mA_G$ admit a continuous parametrization $(U,\bA,a)$ satisfying \eqref{param}. In addition, assume that the parametrization satisfies \eqref{AF_mon}--\eqref{Ag_0}. Then for any $u_0\in L^1(\R^d)$ there exists an entropy weak solution $u\in \mathcal{C}(0,T;L^1(\R))$. Moreover, let $(U^1,\bA^1,a^1)$ and $(U^2,\bA^2,a^2)$ be two equivalent parameterizations, i.e., fulfill \eqref{eqpa}, and let $u^1$ and $u^2$ be two corresponding entropy weak solutions with initial data $u_0$. Then $u^1=u^2$ almost everywhere  in~$\R^{d+1}_+$. In addition, if $a^1$ satisfies \eqref{strictnm} then $v$ is uniquely determined by $v_0$, which is assumed to satisfy $U(v_0)=u_0$.
\label{TT3}
\end{Theorem}

\subsection{Structure of the paper and auxiliary identities}
The paper is organized as follows. The remaining part of this section is devoted to some auxiliary identities frequently used in the paper. In Section~\ref{S2} we provide a proof of Theorem~\ref{TT1}. Then in Section~\ref{S3}, we prove Theorem~\ref{TT2} in detail and also sketch the proof of the generalized version stated in Theorem~\ref{TT3}.

Below we state several identities. The following first set of identities easily follows from the definition of $\chi$, see also \cite{Perthame02},
\begin{align}\label{prop1}
\partial_{\xi}\chi(\xi,u)&=\delta_0(\xi) -\delta_{u}(\xi),\\
\label{prop2}
\int_{\mathbb{R}}f'(\xi)\chi(\xi,u)\dxi&= f(u)-f(0),\\
\int_{\mathbb{R}}|\chi(\xi,v)-\chi(\xi,u)|\dxi&= |u-v|.\label{prop3}
\end{align}
Thus, using the integration by parts formula we can deduce that
\begin{equation}\label{int_p}
\int_{\mathbb{R}} \int_{-\infty}^{\xi} f(s)g'(\xi)\; ds \; d\xi=-\int_{\mathbb{R}} f(\xi)g(\xi) \; d\xi,
\end{equation}
which is valid for all $f\in L^1(\mathbb{R})$ and all $g\in W^{1,\infty}(\mathbb{R})$ satisfying $\lim_{\xi\to \infty}g(\xi)=0$.

\section{Proof of Theorem~\ref{TT1}}\label{S2}
In this section we provide the proof of Theorem~\ref{TT1}.
We skip showing the equivalence of Definition~\ref{D2} and Definition~\ref{D3} since it is almost obvious.
For simplicity we consider now $U$, $\bA$, $a$ to be locally Sobolev functions and focus on the equivalence of Definition~\ref{D1} and Definition~\ref{D2}. First, let us show that any entropy weak solution is also the kinetic solution. For this purpose, we define the distribution $m$ by the following formula
\begin{align}
\left.\begin{aligned}
\langle m, \varphi\rangle:=&-\int_{\R^{d+2}_+}\int_{-\infty}^{\xi}U'(s)\chi(s,v(t,x))\partial_t \varphi(t,x,\xi)\ds\dxi\dx\dt\\
 &-\int_{\R^{d+2}_+} \int_{-\infty}^{\xi}\bA'(s) \chi(s,v(t,x))\cdot \nabla_x\varphi(t,x,xi)\ds\dxi\dx\dt\\
  &+\int_{\R^{d+2}_+}\int_{-\infty}^{\xi}a(s)\partial_{s} \chi(s,v(t,x))\varphi(\xi,t,x)\ds\dxi\dx\dt\\
    &- \int_{\R^{d+1}}\int_{-\infty}^{\xi}U'(s)\chi(s,v_0(x))\varphi(0,x,\xi)\ds\dxi\dx\\
    & -\int_{\mathbb{R}^{d+2}_+} \int_{-\infty}^{\xi} a(0)\delta_0(s) \varphi(t,x,\xi) \ds\dxi\dx \dt.
%\chi(\cdot,u_0)&=\chi(0) &&\textrm{in }\mathbb{R}^{d+1},
\end{aligned}\right.
\label{P1}
\end{align}
Since $\chi$ is a BV function the distribution is well defined. Moreover, it is easily seen from the fact that $v$ and $v_0$ are locally bounded that $m\in (\mathcal{C}^{1}_0(\mathbb{R}^{d+2}))^*$. Next, since $u$ and $v$ belong to $L^{\infty}(0,T;L^{\infty}_{loc}(\mathbb{R}^d))$, we can define a finite $R(t,x):=|v(t,x)|$ (note here that the function $R$ is locally bounded) and then it directly follows from the definition of $\chi$, see \eqref{chi}, that if $\varphi(t,x,\xi)=0$ for all $\xi> -R(t,x)$ then all integrands in \eqref{P1} are identically zero and consequently $\langle m,\varphi\rangle =0$.  On the other hand assuming that $\varphi(t,x,\xi)=0$ for all $\xi\le R(t,x)$, we can easily extend the integration range in all integrals appearing in \eqref{P1} to obtain
\begin{align*}
\left.\begin{aligned}
\langle m, \varphi\rangle=&-\int_{\R^{d+2}_+}\int_{-\infty}^{\infty}U'(s)\chi(s,v(t,x))\partial_t \varphi(t,x,\xi)\ds\dxi\dx\dt\\
 &-\int_{\R^{d+2}_+} \int_{-\infty}^{\infty}\bA'(s) \chi(s,v(t,x))\cdot \nabla_x\varphi(t,x,\xi)\ds\dxi\dx\dt\\
  &+\int_{\R^{d+2}_+}\int_{-\infty}^{\infty}a(s)\partial_{s} \chi(s,v(t,x))\varphi(t,x,\xi)\ds\dxi\dx\dt\\
    &- \int_{\R^{d+1}}\int_{-\infty}^{\infty}U'(s)\chi(s,v_0(x))\varphi(0,x,\xi)\ds\dxi\dx\\
    & -\int_{\mathbb{R}^{d+2}_+} \int_{-\infty}^{\infty} a(0)\delta_0(s) \varphi(t,x,\xi) \ds\dxi\dx \dt\\
    =&-\int_{\R^{d+2}_+}U(v(t,x))\partial_t \varphi(t,x,\xi)+ \bA(v(t,x)) \cdot \nabla_x\varphi(t,x,\xi)\dxi\dx\dt\\
  &-\int_{\R^{d+2}_+}a(v(t,x))\varphi(t,x,\xi)\dxi\dx\dt- \int_{\R^{d+1}}U(v(t,x))\varphi(0,x,\xi)\dx=0,
\end{aligned}\right.
\end{align*}
where for the second equality we used identities \eqref{prop1}--\eqref{prop3} and for the last equality we used \eqref{KruzDF} with the entropy $E(s)=|s\pm k|$ with sufficiently large $k$ depending on the support of $\varphi$.  Note that  such a chosen $E$ is not sufficiently regular, but it is a standard procedure to show with help of density arguments that such entropies may be used. Hence, we see that the distribution $m$ is locally (depending on $(t,x)$) supported (with respect to $\xi$) in a  ball. It is also evident from the definition, just by using integration by parts, that $m$  solves \eqref{E1kBweak} or in case that $U$ and $\bA$ are not sufficiently smooth it solves at least \eqref{E1kBweak2}. Our goal now is to show that $m$ is nonnegative and therefore is a measure. Since we know that $m$ is a distribution, to prove that it is a measure, it  is enough to show that $\langle m, \varphi\rangle \ge 0$ for all $\varphi$'s of the form $\varphi(t,x,\xi):=\psi_1(\xi)\psi_2(t,x)$, where $\psi_1\in \mathcal{D}(\mathbb{R})$ and $\psi_2\in \mathcal{D}(\mathbb{R}^{d+1})$ are nonnegative. Denoting also the convex function $E(s):=-\int_0^s \int_{t}^{\infty}\psi_1(y)\dy\dt$, we see that $E$ is smooth convex function and $E''=\psi_1$.
With this special choice, the identity \eqref{P1} reduces to
\begin{align}
\left.\begin{aligned}
\langle m, \psi_1\psi_2\rangle=&-\int_{\R^{d+2}_+}\int_{-\infty}^{\xi}U'(s)\chi(s,v(t,x))E''(\xi)\partial_t \psi_2(t,x)\ds\dxi\dx\dt\\
 &-\int_{\R^{d+2}_+} \int_{-\infty}^{\xi}\bA'(s) \chi(s,v(t,x))E''(\xi)\cdot \nabla_x\psi_2(t,x)\ds\dxi\dx\dt\\
  &+\int_{\R^{d+2}_+}\int_{-\infty}^{\xi}a(s)\partial_{s} \chi(s,v(t,x))E''(\xi)\psi_2(t,x)\ds\dxi\dx\dt\\
    &- \int_{\R^{d+1}}\int_{-\infty}^{\xi}U'(s)\chi(s,v_0(x))E''(\xi)\psi_2(0,x)\ds\dxi\dx\\
     &-\int_{\mathbb{R}^{d+2}_+} \int_{-\infty}^{\xi} a(0)\delta_0(s) E''(\xi)\psi_2(t,x) \ds\dxi \dx \dt.
%\chi(\cdot,u_0)&=\chi(0) &&\textrm{in }\mathbb{R}^{d+1},
\end{aligned}\right.
\label{P2}
\end{align}
Then using \eqref{int_p} with $g:=E'$ (note that $E'(\infty)=0$), we see that \eqref{P2} can be simplified in the following way
\begin{align}
\left.\begin{aligned}
\langle m, \psi_1\psi_2\rangle=&\int_{\R^{d+2}_+}U'(\xi)\chi(\xi,v(t,x))E'(\xi)\partial_t \psi_2(t,x)\; d\xi\; dx\; dt\\
 &+\int_{\R^{d+2}_+} \bA'(\xi) \chi(\xi,v(t,x))E'(\xi)\cdot \nabla_x\psi_2(t,x)\; d\xi\; dx\; dt\\
  &-\int_{\R^{d+2}_+}a(\xi)\partial_{\xi} \chi(\xi,v(t,x))E'(\xi)\psi_2(t,x)\; d\xi\; dx\; dt\\
    &+\int_{\R^{d+1}}U'(\xi)\chi(\xi,v_0(x))E'(\xi)\psi_2(0,x)\; d\xi\; dx\\
    & +\int_{\mathbb{R}^{d+2}_+} a(0)\delta_0(\xi)E'(\xi)\psi_2(t,x) \; d\xi \; dx \; dt.
%\chi(\cdot,u_0)&=\chi(0) &&\textrm{in }\mathbb{R}^{d+1},
\end{aligned}\right.
\label{P3}
\end{align}
Recalling the definition of the fluxes $Q_U$ and $Q_{\bA}$ and the properties of $\chi$, see \eqref{prop1}--\eqref{prop2},
we see that \eqref{P3} reduces to
\begin{align}
\left.\begin{aligned}
\langle m, \psi_1\psi_2\rangle=&\int_{\R^{d+1}_+}(Q_U(v(t,x))-Q_U(0))\partial_t \psi_2(t,x)\dx\dt\\
 &+\int_{\R^{d+1}_+} (Q_{\bA}(v(t,x))-Q_{\bA}(0))\cdot \nabla_x\psi_2(t,x)\dx\dt\\
  &+\int_{\R^{d+1}_+}a(v(t,x))E'(v(t,x))\psi_2(t,x)\dx\dt\\
    &+\int_{\R^{d}}(Q_U(v_0(x))-Q_U(0))\psi_2(0,x)\dx.
%\chi(\cdot,u_0)&=\chi(0) &&\textrm{in }\mathbb{R}^{d+1},
\end{aligned}\right.
\label{P4}
\end{align}
Finally, using the fact that $\psi_2$ has compact support and integrating by parts with respect to $x$ and $t$ we obtain
\begin{align}
\left.\begin{aligned}
&\langle m, \psi_1\psi_2\rangle=\int_{\R^{d+1}_+}Q_U(v(t,x))\partial_t \psi_2(t,x)\dx\dt+\int_{\R^{d}}Q_U(v_0(x))\psi_2(0,x)\dx\\
 &\quad +\int_{\R^{d+1}_+} Q_{\bA}(v(t,x))\cdot \nabla_x\psi_2(t,x)+a(v(t,x))E'(v(t,x))\psi_2(t,x)\dx\dt,
%\chi(\cdot,u_0)&=\chi(0) &&\textrm{in }\mathbb{R}^{d+1},
\end{aligned}\right.
\label{P5}
\end{align}
which is according to the assumption that $u$ is a weak entropy solution, or more precisely due to  \eqref{KruzDF}, nonnegative and therefore $m$ is a measure.

To prove also the opposite implication, i.e., that any kinetic solution is also a weak entropy solution, it is enough to show the validity of \eqref{KruzDF}. Hence, let $E$ be an arbitrary convex smooth function such that $E'' \in \mathcal{D}(\mathbb{R})$ and $\psi \in \mathcal{D}(\mathbb{R}^{d+1})$ be arbitrary nonnegative function. Consider also a nonnegative $\tau\in \mathcal{D}(-2,2)$ such that $\tau=1$ in $(-1,1)$ and define $\tau_n(s):=\tau(s/n)$. Finally, we set $\varphi(t,x,\xi):=E'(\xi)\psi(t,x)\tau_n(\xi)$ in \eqref{E1kBweak} to obtain (using also the definition of $Q_U$ and $Q_{\bA}$)
\begin{align*}
&\left\langle m(t,x,\xi),\psi(t,x)\partial_{\xi}(E'(\xi)\tau_n(\xi))\right\rangle\\
&=\int_{\mathbb{R}^{d+2}_+} Q'_U(\xi)\chi(\xi,v(t,x))\tau_n(\xi)\partial_t \psi(t,x) \dxi\dx \dt\\
&+\int_{\mathbb{R}^{d+2}_+} Q'_{\bA}(\xi)\chi(\xi,v(t,x))\tau_n(\xi)\cdot \nabla_x \psi(t,x)\dxi\dx \dt\\
&+\int_{\mathbb{R}^{d+1}_+} a(0)E'(0)\psi(t,x) \dx \dt+\int_{\mathbb{R}^{d+2}_+} a'(\xi)\tau_{n}(\xi)E'(\xi) \chi(\xi,v(t,x))\psi(t,x)  \dxi \dx \dt\\
  &+\int_{\mathbb{R}^{d+1}} Q'_U(\xi)\tau_n(\xi)\chi(\xi,v_0(x))\psi(0,x) \dxi\dx\\
  &+\int_{\mathbb{R}^{d+2}_+} a(\xi) \chi(\xi,v(t,x))\psi(t,x)\partial_{\xi}(E'(\xi)\tau_n(\xi))  \dxi \dx\dt.
%\label{P8}
\end{align*}
Therefore, defining the auxiliary fluxes
$$
Q^n_U(\xi):=\int_0^{\xi} Q'_U(s)\tau_n(s)\ds, \quad Q_{\bA}^n(\xi):=\int_0^{\xi}Q'_{\bA}(s)\tau_n(s)\ds
$$
and using the identities  \eqref{prop1}--\eqref{prop2} and integration by parts, we observe that
\begin{align}
\left.\begin{aligned}
&\langle m(t,x,\xi),\psi(t,x) E'(\xi)\tau'_n(\xi)\rangle+\langle m(t,x,\xi),\psi(t,x)E''(\xi)\tau_n(\xi)\rangle\\
&= \int_{\mathbb{R}^{d+1}_+} (Q^n_U(v(t,x))-Q^n_U(0))\partial_t \psi(t,x) \dx \dt\\
&\quad + \int_{\mathbb{R}^{d+1}_+}(Q^n_{\bA}(v(t,x))-Q^n_{\bA}(0))\cdot \nabla_x \psi(t,x)\dx \dt\\
  &\quad +\int_{\mathbb{R}^{d}} (Q^n_U(v_0(t,x))-Q^n_U (0))\psi(0,x) \dx\\
  &\quad +\int_{\mathbb{R}^{d+1}_+} a(v(t,x)) \psi(t,x)E'(v(t,x))\tau_n(v(t,x))   \dx\dt.
\end{aligned}\right.
\label{P10}
\end{align}
Finally, we see that the second term on the left hand side is nonnegative since $m, \psi, \tau_n$ are nonnegative and $E$ is convex. Moreover, since we know that $m$ is locally (with respect to $(t,x)$) compactly supported measure (with respect to $\xi$), we can find a sufficiently large $n$ depending on the support of $\psi$ such that the first term on the left hand side is identically zero (note here that  $\tau'_n$ is supported outside of the ball of radius $n$). Therefore, for all compactly supported $\psi$ and all sufficiently large $n$, the left hand side of \eqref{P10} is nonnegative and using also integration by parts, we find that for such $n$'s
\begin{align}
\left.\begin{aligned}
0&\le  \int_{\mathbb{R}^{d+1}_+} Q^n_U(v(t,x))\partial_t \psi(t,x) +Q^n_{\bA}(v(t,x))\cdot \nabla_x \psi(t,x)\dx \dt\\
  &\quad +\int_{\mathbb{R}^{d}} Q^n_U(v_0(t,x))\psi(0,x) \dx\\
  &\quad +\int_{\mathbb{R}^{d+1}_+} a(v(t,x)) \psi(t,x)E'(v(t,x))\tau_n(v(t,x))   \dx\dt.
\end{aligned}\right.
\label{P11}
\end{align}
Consequently, letting $n\to \infty$ in the above inequality, using that $\tau_n(v(t,x))\to 1$, $Q^n_U(v(t,x)) \to Q_U(v(t,x))$ and $Q_{\bA}^n(v(t,x))\to Q_{\bA}(v(t,x))$, we see that  \eqref{KruzDF} follows.

\section{Further auxiliary tools} \label{S3}
This section is devoted to introduction of two tools that will be used later for the proof of Theorem~\ref{TT2}. The first one is a stability of
%generalization of the Kat\'{o} inequality into the framework of the concept of
\emph{entropy measure valued solution}, which will be used for the uniqueness (and also existence) part of Theorem~\ref{TT2}. The second one is a proper \emph{kinetic approximation} of \eqref{E1g}, which will be used for the existence part of Theorem~\ref{TT2}

\subsection{Entropy measure valued solution and stability inequalities}
We start with the notion of the entropy measure valued solution. Its importance consists in the fact that one can prove the existence of such a solution quite directly just by establishing the a~priori estimates.
\begin{Def}[Entropy measure valued solution]\label{measure}
Let $u_0 \in L^{\infty}_{loc}(\mathbb{R}^d)$. Assume that graphs $\mA_{\bF}$ and $\mA_G$ admit a continuous parametrization \eqref{param}. We say that a Young measure $\nu\in L^{\infty}_{weak}((0,T)\times K; \mathcal{P}_0(\mathbb{R}))$, for any compact $K\subset\R^d$ is an entropy measure valued solution to  \eqref{E1g} and \eqref{param} with the initial data $u_0$ if and only if  for  all smooth convex functions $E:\mathbb{R}\to \mathbb{R}$ such that $E''\in \mathcal{D}(\mathbb{R})$, and for all nonnegative $\varphi \in \mathcal{D}((-\infty,T)\times \mathbb{R}^d)$ there holds
\begin{equation}
\begin{aligned}
&\int_{\mathbb{R}^{d+1}_+} \langle Q_U(\lambda),\nu_{(t,x)}(\lambda)\rangle\partial_t \varphi(t,x) + \langle Q_{\bA}(\lambda),\nu_{(t,x)}(\lambda)\rangle\cdot \nabla_x \varphi(t,x)\dx \dt \\
&\ge -\int_{\mathbb{R}^{d+1}_+} \langle a(\lambda)E'(\lambda),\nu_{(t,x)}(\lambda)\rangle \varphi(t,x)\dx \dt -\int_{\mathbb{R}}Q_U(v_0(x))\varphi(0,x)  \dx, \label{Kruzmeasure}
\end{aligned}
\end{equation}
where  fluxes $Q_U$ and $Q_{\bA}$ are given by \eqref{DFflux} and $v_0$ satisfies $U(v_0)=u_0$ almost everywhere in $\mathbb{R}^d$.
\end{Def}
On this level of ``solution" we  want to pass to the entropy weak solution and furthermore  we also want to show the uniqueness of a solution and its independence of the graph parametrization. As an essential tool used in many cases, we show the following  result on stability of solutions.
\begin{Theorem}[Stability inequalities]\label{KatoT}
Let $\mA_{\bF}$ admit a continuous parametrization $(U,\bA)$. Let two graphs $\mA_{G^1}$ and $\mA_{G^2}$ admit the continuous parametrization $(U,a^1)$ and $(U,a^2)$ respectively. Moreover, let $a^1\le a^2$ everywhere and  $u_0^i \in L^{\infty}_{loc}(\mathbb{R}^d)$. Then for any entropy measure valued  solution $\nu^{i}$ corresponding to $(U,\bA,a^i,v_0^i)$ with $i=1,2$ and arbitrary nonnegative $\varphi\in \mathcal{D}((-\infty,T)\times \R^d)$ there holds
\begin{equation}
\begin{aligned}
&-\int_0^T\int_{\R^d}\langle(U(\lambda)-U(k))_+,\nu^{1}_{(t,x)}(\lambda)\otimes\nu^{2}_{(t,x)}(k)\rangle  \partial_t \varphi(t,x) \dx \dt \\
&\quad -\int_0^T\int_{\R^d}\langle (\bA(\lambda)-\bA(k))\bbbone_{\{\lambda \ge k\}},\nu^{1}_{(t,x)}(\lambda)\otimes\nu^{2}_{(t,x)}(k)\rangle \cdot \nabla_x \varphi(t,x)  \dx \dt\\
& \le \int_0^T\int_{\R^d}\langle ((a^1(\lambda)-a^{2}(k))\bbbone_{\{\lambda \ge k\}} ,\nu^{1}_{(t,x)}(\lambda)\otimes \nu^{2}_{(t,x)}(k)\rangle \varphi(t,x)  \dx \dt\\
&\quad +\int_{\R^d}(U(v_0^1(x))-U(v_0^2(x)))_+\varphi(0,x)  \dx.\label{resultin1}
\end{aligned}
\end{equation}
Moreover, assume that the graphs $\mA_{\bF}$ and $\mA_{G}$ admit two equivalent continuous parametrizations $(U^1,\bA^1,a^1)$ and $(U^2,\bA^2,a^2)$ fulfilling \eqref{eqpa}. Let $\nu^1$ and $\nu^2$ be two  entropy measure valued solutions corresponding to the initial data $v_0^1$ and $v_0^2$. Then for all $\varphi \in \mathcal{D}((-\infty,T)\times \R^d)$ there holds
\begin{equation}
\begin{aligned}
&-\int_0^T \int_{\R^d} \langle(U^1(\lambda)-U^2(k))_+,\nu^{1}_{(t,x)}(\lambda)\otimes\nu^{2}_{(t,x)}(k)\rangle \partial_t\varphi(t,x) \dx \dt\\
&\quad- \int_0^T \int_{\R^d}  \langle (\bA^1(\lambda)-\bA^2(k))\bbbone_{\{\lambda \ge \tilde{U}(k)\}},\nu^{1}_{(t,x)}(\lambda)\otimes\nu^{2}_{(t,x)}(k)\rangle \cdot \nabla_x \varphi (t,x)\dx \dt\\
& \le \int_0^T \int_{\R^d} \langle (a^{1}(\lambda)-a^2(k))\bbbone_{\{\lambda \ge \tilde{U}(k)\}},\nu^{1}_{(t,x)}(\lambda)\otimes \nu^2_{(t,x)}(k)\rangle \varphi(t,x) \dx\dt\\
&\quad + \int_{\R^d} (U^1(v_0^1(x))-U^2(v_0^2(x)))_+\varphi(0,x) \dx.
\label{resultin2}
\end{aligned}
\end{equation}
In addition, for arbitrary entropy measure valued solution, the initial data are attained in the following sense
\begin{equation}
\label{indatameasure}
\lim_{\tau \to 0_+} \frac{1}{\tau} \int_0^{\tau} \int_K \langle |U(\lambda)-U(v_0(x))|,\nu_{(t,x)}(\lambda)\rangle \dx \dt=0,
\end{equation}
where $K\subset \mathbb{R}^d$ is an arbitrary compact set.
\end{Theorem}
\begin{proof}
We start the proof with showing the essential property  \eqref{indatameasure}, that will be also used later.  We first focus on weak attainment of the initial data. Since the measure is locally bounded, we know that for all  bounded balls $B\subset \mathbb{R}^d$ there exists $R$ such that $\supp \nu_{(t,x)}\subset [-R,R]$ whenever $x \in B$. Hence, defining\footnote{Although such an entropy is not a smooth function, the use of this entropy can be easily justified by a proper approximative scheme, see also the next subsection.}  the entropy $E(\lambda):=|\lambda \pm R|$, we see that for all $x\in B$, there holds $E(\lambda)=\pm(\lambda \pm R)$ for all $\lambda$ belonging to the support of $\nu_{(t,x)}$. Consequently, with this choice of the entropies, we can deduce from \eqref{Kruzmeasure} that for an arbitrary $\varphi \in \mathcal{D}((-\infty,T)\times B)$
\begin{equation*}
\begin{aligned}
&\int_{\mathbb{R}^{d+1}_+} \langle U(\lambda),\nu_{(t,x)}(\lambda)\rangle\partial_t \varphi(t,x) + \langle \bA(\lambda),\nu_{(t,x)}(\lambda)\rangle\cdot \nabla_x \varphi(t,x)\dx \dt \\
&= -\int_{\mathbb{R}^{d+1}_+} \langle a(\lambda),\nu_{(t,x)}(\lambda)\rangle \varphi(t,x)\dx \dt -\int_{\mathbb{R}}U(v_0(x))\varphi(0,x)  \dx.
\end{aligned}
\end{equation*}
Due to the weak star density, we see that the above identity holds even for Lipschitz functions $\varphi$. Thus, for an arbitrary $\psi \in \mathcal{D}(B)$ we can set $\varphi(t,x):=\psi(x)(1-\frac{t}{\tau})_+$ to obtain
\begin{equation*}
\begin{aligned}
&\lim_{\tau \to 0_+}\left|\frac{1}{\tau} \int_0^{\tau}\int_{B} \langle U(\lambda),\nu_{(t,x)}(\lambda)\rangle \psi(x)\dx \dt - \int_{B}U(v_0)\psi(x)\dx \right| =0.
\end{aligned}
\end{equation*}
Consequently, defining $u(t,x):=\langle U(\lambda),\nu_{(t,x)}(\lambda)\rangle$, we see that (note here that for almost all times $t\in (0,T)$, the function $u$ is uniformly bounded in $B$)
\begin{equation}\label{weak_att}
\frac{1}{\tau}\int_0^{\tau} u(t) \dt \rightharpoonup U(v_0) \qquad \textrm{weakly in } L^2(B),
\end{equation}
where we considered the Bochner integral. On the other hand, since $U$ is nondecreasing, the function $E(\lambda):=2\int_0^{\lambda} U(s)\; \ds$ is a possible entropy in  \eqref{Kruzmeasure} and $Q_U(\lambda):=U^2(\lambda)$ is the corresponding flux. Therefore, using this entropy, and also the boudendess of the Young measure in $B$, we can repeat step by step the previous procedure and deduce that for all nonnegative functions $\psi\in  \mathcal{D}(\mathbb{R}^d)$
$$
\limsup_{\tau \to 0_+} \frac{1}{\tau}\int_0^{\tau}\int_{\mathbb{R}^d}\langle U^2(\lambda),\nu_{(t,x)}(\lambda)\rangle \psi(x)\dx \dt\le \int_{\mathbb{R}^d}U^2(v_0(x))\psi(x)\dx,
$$
which by the density result and also due to the local  boundedness of the Young  measure also implies
\begin{equation}\label{limsup_att}
\limsup_{\tau \to 0_+} \frac{1}{\tau}\int_0^{\tau}\int_{B}\langle U^2(\lambda),\nu_{(t,x)}(\lambda)\rangle \dx \dt\le \int_{B}U^2(v_0(x))\dx.
\end{equation}
Using finally the H\"{o}lder and the Jensen inequality, we observe that
\begin{equation*}
\begin{split}
\lim_{\tau \to 0_+} &\frac{1}{\tau} \int_0^{\tau} \int_K \langle |U(\lambda)-U(v_0(x))|,\nu_{(t,x)}(\lambda)\rangle \dx \dt\\
&\le \lim_{\tau \to 0_+} \frac{1}{\tau} \int_0^{\tau} \int_K \sqrt{\langle |U(\lambda)-U(v_0(x))|^2,\nu_{(t,x)}(\lambda)\rangle} \dx \dt\\
&\le |K|^{\frac12} \sqrt{\lim_{\tau \to 0_+}  \frac{1}{\tau} \int_0^{\tau} \int_K \langle |U(\lambda)-U(v_0(x))|^2,\nu_{(t,x)}(\lambda)\rangle \dx \dt}\\
&= |K|^{\frac12} \left(\lim_{\tau \to 0_+}  \frac{1}{\tau} \int_0^{\tau} \int_K \langle U^2(\lambda),\nu_{(t,x)}(\lambda)\rangle+ U^2(v_0(x)) \dx \dt \right.\\
&\qquad \; \qquad -2\int_K U(v_0(x)) \left( \frac{1}{\tau}\int_0^{\tau} u(t,x)\dt \right)\dx \Big)^{\frac12}\\
&\overset{\eqref{weak_att},\eqref{limsup_att}}\le |K|^{\frac12} \left(\int_K 2U^2(v_0(x)) \dx - \int_K 2 U^2(v_0(x))\dx \right)^{\frac12}=0,
\end{split}
\end{equation*}
which finishes the proof of  \eqref{indatameasure}.

Next, we focus on the stability  inequality.
For the convex  function $E^+(\lambda):=\max(\lambda,0)$ we consider its approximation $E^+_n\in \mathcal{C}^1(\mathbb{R})$  given as
$$
E^+_n(\lambda):=\left\{ \begin{aligned}
&0 &&\textrm{for } \lambda \le 0,\\
&n\lambda^2 &&\textrm{for } \lambda \in \left[0,\frac{1}{2n}\right],\\
&\lambda-\frac{1}{4n} &&\textrm{for } \lambda > \frac{1}{2n}.
\end{aligned}\right.
$$
Similarly, we also introduce $E^{-}(\lambda):=E^{+}(-\lambda)$ and $E_n^{-}(\lambda):=E^{+}_n(-\lambda)$ and finally we define $E:=E^+ + E^{-}$ and $E_n:= E_n^+ + E_n^{-}$. Note that all these functions are convex and continuous and furthermore, the mollified functions are of the class $\mathcal{C}^{1}$. Therefore, these mollified functions are   proper entropies in \eqref{Kruzmeasure} (although \eqref{Kruzmeasure} is defined for smooth entropies, it follows from the density that it holds  even  for entropies only with a continuous derivative).

First, we prove the inequality \eqref{resultin1}. Let us consider $(U, \bA,a^1,v^1_0)$, $(U, \bA, a^2, v_0^2)$  and corresponding  entropy  measure valued solutions $\nu^1$ and $\nu^2$ to \eqref{Kruzmeasure}. For any Young measure $\nu^{i}$, there exists $\nu^{i,\varepsilon}\in \mathcal{C}^{\infty}_{loc}((\varepsilon,T-\varepsilon)\times \R^d; \mathcal{P}_0(\R))$
such that
$$
\langle g(\lambda),\nu^{i,\varepsilon}_{(t,x)}\rangle= \int_{\R^{d+1}}\langle g(\lambda), \nu^{i}_{(s,y)} \rangle \eta_{\varepsilon}(t-s,x-y)\dy \ds
$$
for any $g\in \mathcal{C}(\R)$ and $(t,x)\in (\varepsilon,T-\varepsilon)\times \R$, where $\eta_{\varepsilon}$ is the standard convolution kernel of radius $\varepsilon$.

Thus, considering for arbitrary $\tau\in (\varepsilon, T-\varepsilon)$ and $y\in \mathbb{R}^d$, the test function $\varphi(t,x):=\eta_{\varepsilon}(\tau-t,y-x)$ in \eqref{Kruzmeasure}, we obtain for all admissible entropies and fluxes the following inequality
\begin{equation}
\begin{aligned}
&\partial_{\tau}\langle Q_U(\lambda),\nu^{i,\varepsilon}_{(\tau,y)}(\lambda)\rangle + \diver_y \langle Q_{\bA}(\lambda),\nu^{i,\varepsilon}_{(\tau,y)}(\lambda)\rangle\le \langle a^{i}(\lambda)E'(\lambda),\nu^{i,\varepsilon}_{(\tau,y)}(\lambda)\rangle  \label{Kruzmeasurepoi}
\end{aligned}
\end{equation}
valid everywhere in $(\varepsilon,T-\varepsilon)\times \mathbb{R}^d$.

Next, for arbitrary $k\in \R$, we consider the the following entropies $E(\lambda):=E^+_n (\lambda -k)$ for which we  find the corresponding fluxes in the following way
$$
\begin{aligned}
Q^n_U(\lambda)&=\int_{-\infty}^\lambda U'(s)(E^+_n)'(s-k)\ds=(U(\lambda)-U(k))_+ + P_1^n(\lambda,k),\\
Q^n_{\bA}(\lambda)&=\int_{-\infty}^\lambda \bA'(s)(E^+_n)'(s-k)\ds=(\bA(\lambda)-\bA(k))\bbbone_{\lambda \ge k} + P_2^n(\lambda,k),
\end{aligned}
$$
where $P_1^n$ and $P_2^n$ converge locally (in $\mathbb{R}\times \mathbb{R}$) uniformly to zero as $n\to \infty$. Therefore, using these entropies in \eqref{Kruzmeasurepoi} (we just replace $(\tau,y)$ by $(t,x)$), we obtain for $i=1$
\begin{equation}
\begin{aligned}
&\partial_t\langle(U(\lambda)-U(k))_+,\nu^{1,\varepsilon}_{(t,x)}(\lambda)\rangle + \diver_x \langle (\bA(\lambda)-\bA(k))\bbbone_{\{\lambda \ge k\}},\nu^{1,\varepsilon}_{(t,x)}(\lambda)\rangle\\
&\quad \le \langle a^{1}(\lambda)(2n(\lambda-k)\bbbone_{\{\lambda \in (k,k+\frac{1}{2n})\}}+\bbbone_{\{\lambda \ge k+\frac{1}{2n}\}} ,\nu^{1,\varepsilon}_{(t,x)}(\lambda)\rangle  + P^{n,\varepsilon}(t,x,k),\label{Kruzmeasurepoi2}
\end{aligned}
\end{equation}
where $P^{n,\varepsilon}(t,x,k)\to 0$ in $\mathcal{C}_{loc}((\varepsilon, T-\varepsilon)\times \R^{d+1})$. In a very similar way, we use the entropy $E(k):=E^{-}_n(k-\lambda)$ and conclude for $i=2$ that
\begin{equation}
\begin{aligned}
&\partial_t\langle(U(\lambda)-U(k))_+,\nu^{2,\varepsilon}_{(t,x)}(k)\rangle + \diver_x \langle (\bA(\lambda)-\bA(k))\bbbone_{\{\lambda \ge k\}},\nu^{2,\varepsilon}_{(t,x)}(k)\rangle\\
&\; \le -\langle a^{2}(k)(2n(\lambda-k)\bbbone_{\{\lambda \in (k,k+\frac{1}{2n})\}}+\bbbone_{\{\lambda \ge k+\frac{1}{2n}\}} ,\nu^{2,\varepsilon}_{(t,x)}(k)\rangle  + P^{n,\varepsilon}(t,x,\lambda).\label{Kruzmeasurepoi3}
\end{aligned}
\end{equation}
Finally, we apply  $\nu^{2,\varepsilon}_{(t,x)}(k)$ to \eqref{Kruzmeasurepoi2} and  $\nu^{1,\varepsilon}_{(t,x)}(\lambda)$ to \eqref{Kruzmeasurepoi3}. Such a procedure is possible since both regularized measures are nonnegative and smooth with respect to $(t,x)$. Then we sum the resulting  inequalities and using the chain rule (which is now possible due to the smoothness of all quantities with respect to $(t,x)$), we obtain that for all $t\in (\varepsilon,T-\varepsilon)$ and all $x\in \mathbb{R}^d$ there holds
\begin{equation*}
\begin{aligned}
&\partial_t\langle(U(\lambda)-U(k))_+,\nu^{1,\varepsilon}_{(t,x)}(\lambda)\otimes\nu^{2,\varepsilon}_{(t,x)}(k)\rangle \\
&\quad+ \diver_x \langle (\bA(\lambda)-\bA(k))\bbbone_{\{\lambda \ge k\}},\nu^{1,\varepsilon}_{(t,x)}(\lambda)\otimes\nu^{2,\varepsilon}_{(t,x)}(k)\rangle\\
& \le \left\langle (a^1(\lambda)-a^{2}(k))(2n(\lambda-k)\bbbone_{\{\lambda \in (k,k+\frac{1}{2n})\}}+\bbbone_{\{\lambda \ge k+\frac{1}{2n}\}}, \right. \\
&\qquad \left. \nu^{1,\varepsilon}_{(t,x)}(\lambda)\otimes \nu^{2,\varepsilon}_{(t,x)}(k)\right\rangle + Y^{n,\varepsilon}(t,x),\label{Kruzmeasurepoi4}
\end{aligned}
\end{equation*}
where $Y^{n,\varepsilon} \to 0$ in $\mathcal{C}_{loc} ((\varepsilon, T-\varepsilon)\times \mathbb{R}^d)$  as $n\to \infty$. At this place, we finally use the key assumption, i.e., the fact  that $a^1\le a^2$, and we see that  the right hand side can be simplified, which leads to the following inequality
\begin{equation}
\begin{aligned}
&\partial_t\langle(U(\lambda)-U(k))_+,\nu^{1,\varepsilon}_{(t,x)}(\lambda)\otimes\nu^{2,\varepsilon}_{(t,x)}(k)\rangle \\
&\quad+ \diver_x \langle (\bA(\lambda)-\bA(k))\bbbone_{\{\lambda \ge k\}},\nu^{1,\varepsilon}_{(t,x)}(\lambda)\otimes\nu^{2,\varepsilon}_{(t,x)}(k)\rangle\\
& \le \left\langle (a^2(\lambda)-a^{2}(k))(2n(\lambda-k)\bbbone_{\{\lambda \in (k,k+\frac{1}{2n})\}}+\bbbone_{\{\lambda \ge k+\frac{1}{2n}\}}, \right.\\
&\qquad \left. \nu^{1,\varepsilon}_{(t,x)}(\lambda)\otimes \nu^{2,\varepsilon}_{(t,x)}(k)\right\rangle + Y^{n,\varepsilon}(t,x).\label{Kruzmeasurepoi5}
\end{aligned}
\end{equation}
Since $a^2$ is continuous,  we can easily let $n\to \infty$ in \eqref{Kruzmeasurepoi5} to deduce the final inequality
\begin{equation}
\begin{aligned}
&\partial_t\langle(U(\lambda)-U(k))_+,\nu^{1,\varepsilon}_{(t,x)}(\lambda)\otimes\nu^{2,\varepsilon}_{(t,x)}(k)\rangle \\
&\quad+ \diver_x \langle (\bA(\lambda)-\bA(k))\bbbone_{\{\lambda \ge k\}},\nu^{1,\varepsilon}_{(t,x)}(\lambda)\otimes\nu^{2,\varepsilon}_{(t,x)}(k)\rangle\\
& \le \langle ((a^2(\lambda)-a^{2}(k))\bbbone_{\{\lambda \ge k\}} ,\nu^{1,\varepsilon}_{(t,x)}(\lambda)\otimes \nu^{2,\varepsilon}_{(t,x)}(k)\rangle.\label{Kruzmeasurepoif}
\end{aligned}
\end{equation}
Then for arbitrary nonnegative  $\varphi\in \mathcal{D}((-\infty,T)\times \R^{d+1})$ and nonnegative $\tau \in \mathcal{D}(0,\infty)$, we can find sufficiently small $\varepsilon_0>0$ such that for all $\varepsilon \in (0,\varepsilon_0)$ the support of $\varphi \tau$ belongs to $(\varepsilon, T-\varepsilon)\times K$, where $K\subset \mathbb{R}^d$ is a compact set. Consequently,  $\varphi \tau$ is an appropriate  test function in \eqref{Kruzmeasurepoif} thus  integrating  by parts and   letting $\varepsilon \to 0_+$ gives
\begin{equation}
\begin{aligned}
&-\int_0^T\int_{\R^d}\langle(U(\lambda)-U(k))_+,\nu^{1}_{(t,x)}(\lambda)\otimes\nu^{2}_{(t,x)}(k)\rangle \tau(t) \partial_t \varphi(t,x) \dx \dt \\
&\; -\int_0^T\int_{\R^d}\langle (\bA(\lambda)-\bA(k))\bbbone_{\{\lambda \ge k\}},\nu^{1}_{(t,x)}(\lambda)\otimes\nu^{2}_{(t,x)}(k)\rangle \cdot \nabla_x \varphi(t,x) \tau(t) \dx \dt\\
& \le \int_0^T\int_{\R^d}\langle ((a^2(\lambda)-a^{2}(k))\bbbone_{\{\lambda \ge k\}} ,\nu^{1}_{(t,x)}(\lambda)\otimes \nu^{2}_{(t,x)}(k)\rangle \varphi(t,x) \tau (t)\dx \dt\\
&\quad +\int_0^T\int_{\R^d}\langle(U(\lambda)-U(k))_+,\nu^{1}_{(t,x)}(\lambda)\otimes\nu^{2}_{(t,x)}(k)\rangle  \varphi(t,x) \tau'(t) \dx \dt .\label{find}
\end{aligned}
\end{equation}
Next, due to the weak star density, we see that the above inequality holds also for Lipschitz $\tau$ fulfilling $\tau(0)=0$. Thus, setting  $\tau(t):=\min(nt,1)$ in \eqref{find} and  letting $n\to \infty$ we find that
\begin{equation}
\begin{aligned}
&-\int_0^T\int_{\R^d}\langle(U(\lambda)-U(k))_+,\nu^{1}_{(t,x)}(\lambda)\otimes\nu^{2}_{(t,x)}(k)\rangle  \partial_t \varphi(t,x) \dx \dt \\
&\quad -\int_0^T\int_{\R^d}\langle (\bA(\lambda)-\bA(k))\bbbone_{\{\lambda \ge k\}},\nu^{1}_{(t,x)}(\lambda)\otimes\nu^{2}_{(t,x)}(k)\rangle \cdot \nabla_x \varphi (t,x) \dx \dt\\
& \le \int_0^T\int_{\R^d}\langle ((a^2(\lambda)-a^{2}(k))\bbbone_{\{\lambda \ge k\}} ,\nu^{1}_{(t,x)}(\lambda)\otimes \nu^{2}_{(t,x)}(k)\rangle \varphi(t,x) \dx \dt\\
&\quad +\limsup_{n\to \infty}n\int_0^{\frac{1}{n}}\int_{\R^d}\langle(U(\lambda)-U(k))_+,\nu^{1}_{(t,x)}(\lambda)\otimes\nu^{2}_{(t,x)}(k)\rangle  \varphi(t,x) \dx \dt. \label{find2}
\end{aligned}
\end{equation}
Finally, to estimate the limes superior in the last expression, we use the already proven attainment of  the initial condition \eqref{indatameasure}. Indeed, using the fact that $v_0^i$ are independent  of  $t$ and that $\nu^{i}$ are Young measures, we  estimate the last term as follows (here $K$ is a compact set depending on the support of $\varphi$ and $C$ is a constant depending on  $\|\varphi\|_{\infty}$)
$$
\begin{aligned}
&\limsup_{n\to \infty}n\int_0^{\frac{1}{n}}\int_{\R^d}\langle(U(\lambda)-U(k))_+,\nu^{1}_{(t,x)}(\lambda)\otimes\nu^{2}_{(t,x)}(k)\rangle  \varphi (t,x) \dx \dt\\
& \le C\limsup_{n\to \infty}n\int_0^{\frac{1}{n}}\int_{K}\langle (U(\lambda)-U(v_0^1))_+ ,\nu^{1}_{(t,x)}(\lambda)\otimes\nu^{2}_{(t,x)}(k)\rangle  \dx \dt\\
&\quad+C\limsup_{n\to \infty}n\int_0^{\frac{1}{n}}\int_{K}\langle (U(v_0^2)-U(k))_+ ,\nu^{1}_{(t,x)}(\lambda)\otimes\nu^{2}_{(t,x)}(k)\rangle  \dx \dt\\
&\quad+\limsup_{n\to \infty}n\int_0^{\frac{1}{n}}\int_{\R^d}\langle (U(v_0^1)-U(v_0^2))_+ ,\nu^{1}_{(t,x)}(\lambda)\otimes\nu^{2}_{(t,x)}(k)\rangle  \varphi(t,x) \dx \dt\\
& = C\limsup_{n\to \infty}n\int_0^{\frac{1}{n}}\int_{K}\langle (U(\lambda)-U(v_0^1))_+ ,\nu^{1}_{(t,x)}(\lambda)\rangle   \dx \dt\\
&\quad+C\limsup_{n\to \infty}n\int_0^{\frac{1}{n}}\int_{K}\langle (U(v_0^2)-U(k))_+ ,\nu^{2}_{(t,x)}(k)\rangle \dx \dt\\
&\quad+\int_{\R^d}(U(v_0^1)-U(v_0^2))_+ \varphi(0,x) \dx \overset{\eqref{indatameasure}}=\int_{\R^d}(U(v_0^1)-U(v_0^2))_+ \varphi(0,x) \dx.
\end{aligned}
$$
Thus, substituting this inequality into \eqref{find2} we get \eqref{resultin1}.

The second step of the proof focuses on the validity of \eqref{resultin2}. Let us consider two equivalent parameterizations $(U^i, \bA^i, a^i)$ with $i=1,2$ , i.e., satisfying \eqref{eqpa} with some strictly increasing continuous function $\tilde U$ and let $\nu^i$ be two   entropy  measure valued solutions  to \eqref{Kruzmeasure} corresponding to $v^i_0$. We proceed in what follows in a very similar way as before but for simplicity (since it was already shown above), we omit at this moment the mollification with respect to $(t,x)$ and consider all inequalities formally. Thus, we consider for $i=1$ the inequality \eqref{Kruzmeasure} with  the entropy $E(\lambda):=E^+_n (\lambda -\beta_1)$, where $\beta_1\in \mathbb{R}$ and $n\in \mathbb{R}_+$ are arbitrary, and with this setting we obtain
\begin{equation}
\begin{aligned}
&\partial_t\langle(U^1(\lambda)-U^1(\beta_1))_+,\nu^{1}_{(t,x)}(\lambda)\rangle \! + \!\diver_x \langle (\bA^1(\lambda)-\bA^1(\beta_1))\bbbone_{\{\lambda \ge \beta_1\}},\nu^{1}_{(t,x)}(\lambda)\rangle\\
&\le \langle a^{1}(\lambda)(2n(\lambda-\beta_1)\bbbone_{\{\lambda \in (\beta_1,\beta_1+\frac{1}{2n})\}}+\bbbone_{\{\lambda \ge \beta_1+\frac{1}{2n}\}} ,\nu^{1}_{(t,x)}\rangle  + P_1^n(\beta_1),\label{Kruzp2}
\end{aligned}
\end{equation}
where $P_1^n(\beta_1)\to 0$ in $\mathcal{C}_{loc}(\mathbb{R})$ as $n\to \infty$. In a very similar way, for an arbitrary $\beta_2\in \R$ and $m\in \mathbb{R}_+$ (note here that we do not require $m=n$), we use the entropy $E(k):=E^{-}_n(k-\beta_2)$ for $(U^2,\bA^2,a^2)$ and conclude that
\begin{equation}
\begin{aligned}
&\partial_t\langle(U^2(\beta_2)-U^2(k))_+,\nu^{2}_{(t,x)}(k)\rangle + \diver_x \langle (\bA^2(\beta_2)-\bA^2(k))\bbbone_{\{\beta_2 \ge k\}},\nu^{2}_{(t,x)}(k)\rangle\\
&\le -\langle a^{2}(k)(2m(\beta_2-k)\bbbone_{\{\beta_2 \in (k,k+\frac{1}{2m})\}}+\bbbone_{\{\beta_2 \ge k+\frac{1}{2m}\}} ,\nu^{2}_{(t,x)}(k)\rangle  + P_2^m(\beta_2).\label{Kruzp3}
\end{aligned}
\end{equation}
Next,  for an arbitrary $k\in \R$, we set $\beta_1:=\tilde{U}(k)$ and consider an  arbitrary strictly positive function $n(k)\in \mathcal{C}(\R)$ in \eqref{Kruzp2}. With the help of  \eqref{eqpa}, the inequality \eqref{Kruzp2} can be rewritten as
\begin{equation}
\begin{aligned}
&\partial_t\langle(U^1(\lambda)-U^2(k))_+,\nu^{1}_{(t,x)}(\lambda)\rangle + \diver_x \langle (\bA^1(\lambda)-\bA^2(k))\bbbone_{\{\lambda \ge \tilde{U}(k)\}},\nu^{1}_{(t,x)}(\lambda)\rangle\\
& \le \langle a^{1}(\lambda)(2n(k)(\lambda-\tilde{U}(k))\bbbone_{\{\lambda \in (\tilde{U}(k),\tilde{U}(k)+\frac{1}{2n(k)})\}}+\bbbone_{\{\lambda \ge \tilde{U}(k)+\frac{1}{2n(k)}\}} ,\nu^{1}_{(t,x)}(\lambda)\rangle \\
&\qquad  + P_1^{n(k)}(\tilde{U}(k)).\label{Kruzp2a}
\end{aligned}
\end{equation}
Similarly,  setting $\beta_2:=(\tilde{U})^{-1}(\lambda)$ for arbitrary $\lambda\in \R$ and considering arbitrary positive function $m\in \mathcal{C}(\R)$ in \eqref{Kruzp3}, we obtain due to the strict monotonicity of $\tilde{U}$
\begin{equation}
\begin{aligned}
&\partial_t\langle(U^1(\lambda)-U^2(k))_+,\nu^{2}_{(t,x)}(k)\rangle + \diver_x \langle (\bA^1(\lambda)-\bA^2(k))\bbbone_{\{\lambda \ge \tilde{U}(k)\}},\nu^{2}_{(t,x)}(k)\rangle\\
& \le -\left\langle a^{2}(k)(2m(\lambda)((\tilde{U})^{-1}(\lambda)-k)\bbbone_{\{(\tilde{U})^{-1}(\lambda) \in (k,k+\frac{1}{2m(\lambda)})\}}\right.\\
&\qquad \; \qquad \left.+\bbbone_{\{(\tilde{U})^{-1}(\lambda) \ge k+\frac{1}{2m(\lambda)}\}} ,\nu^{2}_{(t,x)}(k)\right\rangle \\
 &\qquad + P_2^{m(\lambda)}((\tilde{U})^{-1}(\lambda)).\label{Kruzp3a}
\end{aligned}
\end{equation}
Then, applying $\nu^{2}_{(t,x)}(k)$ to \eqref{Kruzp2a} and $\nu^{1}_{(t,x)}(\lambda)$ to \eqref{Kruzp3a} and summing the resulting inequalities, we obtain
\begin{equation}
\begin{aligned}
&\partial_t\langle(U^1(\lambda)-U^2(k))_+,\nu^{1}_{(t,x)}(\lambda)\otimes\nu^{2}_{(t,x)}(k)\rangle \\
&\quad+ \diver_x \langle (\bA^1(\lambda)-\bA^2(k))\bbbone_{\{\lambda \ge \tilde{U}(k)\}},\nu^{1}_{(t,x)}(\lambda)\otimes\nu^{2}_{(t,x)}(k)\rangle\\
& \le \left\langle a^{1}(\lambda)(2n(k)(\lambda-\tilde{U}(k))\bbbone_{\{\lambda \in (\tilde{U}(k),\tilde{U}(k)+\frac{1}{2n(k)})\}} \right. \\
&\quad \left.+\bbbone_{\{\lambda \ge \tilde{U}(k)+\frac{1}{2n(k)}\}}) ,\nu^{1}_{(t,x)}(\lambda)\otimes \nu^2_{(t,x)}(k)\right\rangle\\
&-\left\langle a^{2}(k)(2m(\lambda)((\tilde{U})^{-1}(\lambda)-k)\bbbone_{\{(\tilde{U})^{-1}(\lambda) \in (k,k+\frac{1}{2m(\lambda)})\}}\right. \\
&\quad \left. +\bbbone_{\{(\tilde{U})^{-1}(\lambda) \ge k+\frac{1}{2m(\lambda)}\}}) ,\nu^1_{(t,x)}(\lambda)\otimes \nu^{2}_{(t,x)}(k)\right\rangle\\
&\quad +Y(n,m),\label{Kruzp4}
\end{aligned}
\end{equation}
where $Y(n,m) \to 0$ provided that the functions $n(k)\to \infty$ and $m(\lambda)\to \infty$ in $\mathcal{C}_{loc}(\R)$ (note that for this we use the fact that $\nu^i$ are compactly supported). Next, we focus on the limit with respect to $n$ in the terms on the right hand side. Using \eqref{eqpa}, we see that $a^2(k)=a^1(\tilde{U}(k))$ and consequently, we have
\begin{equation*}
\begin{aligned}
&a^{1}(\lambda)(2n(k)(\lambda-\tilde{U}(k))\bbbone_{\{\lambda \in (\tilde{U}(k),\tilde{U}(k)+\frac{1}{2n(k)})\}}+\bbbone_{\{\lambda \ge \tilde{U}(k)+\frac{1}{2n(k)}\}})\\
& - a^{2}(k)(2m(\lambda)((\tilde{U})^{-1}(\lambda)-k)\bbbone_{\{(\tilde{U})^{-1}(\lambda) \in (k,k+\frac{1}{2m(\lambda)})\}}+\bbbone_{\{(\tilde{U})^{-1}(\lambda) \ge k+\frac{1}{2m(\lambda)}\}})\\
&=(a^{1}(\lambda)-a^1(\tilde{U}(k))\left(2n(k)(\lambda-\tilde{U}(k))\bbbone_{\{\lambda \in (\tilde{U}(k),\tilde{U}(k)+\frac{1}{2n(k)})\}}\right.\\
&\quad \; \quad\left.+\bbbone_{\{\lambda \ge \tilde{U}(k)+\frac{1}{2n(k)}\}}\right)\\
&  +a^2(k)\left(2n(k)(\lambda-\tilde{U}(k))\bbbone_{\{\lambda \in (\tilde{U}(k),\tilde{U}(k)+\frac{1}{2n(k)})\}}+\bbbone_{\{\lambda \ge \tilde{U}(k)+\frac{1}{2n(k)}\}})\right.\\
&\quad \; \quad\left. - 2m(\lambda)((\tilde{U})^{-1}(\lambda)-k)\bbbone_{\{(\tilde{U})^{-1}(\lambda) \in (k,k+\frac{1}{2m(\lambda)})\}}-\bbbone_{\{(\tilde{U})^{-1}(\lambda) \ge k+\frac{1}{2m(\lambda)}\}}\right)\\
&=:Z_1^{n,m}(\lambda,k)+Z_2^{n,m}(\lambda,k).
\end{aligned}
\end{equation*}
We see that if $n,m \to \infty \in \mathcal{C}_{loc}(\R)$ then
\begin{equation}\label{Zunif}
Z_1^{m,n}(\lambda,k) \to (a^{1}(\lambda)-a^1(\tilde{U}(k))\bbbone_{\{\lambda \ge \tilde{U}(k)\}} \qquad \textrm{ in } \mathcal{C}_{loc}(\R^2).
\end{equation}
In order to identify also the limit of $Z_2^{n,m}$, we define a special sequence of function $n$ and $m$ in the following way
$$
\begin{aligned}
m(\lambda)&:=m \to \infty &&\textrm{as } m\to \infty,\\
n(k)&:=\frac{1}{2(\tilde{U}(k+\frac{1}{2m})-\tilde{U}(k))}\to \infty &&\textrm{as } m\to \infty.
\end{aligned}
$$
The reason for such a choice is that
$$
\tilde{U}\left(k+\frac{1}{2m}\right) = \tilde{U}(k)+\frac{1}{2n(k)}
$$
and consequently $\bbbone_{\{(\tilde{U})^{-1}(\lambda) \ge k+\frac{1}{2m(\lambda)}\}}= \bbbone_{\{\lambda \ge \tilde{U}(k)+\frac{1}{2n(k)}\}}$. Hence, denoting also $\bbbone:=\bbbone_{\{\lambda \in (\tilde{U}(k),\tilde{U}(k+\frac{1}{2m}))\}}$, we see that $Z_2^{n,m}$ can be estimated as
\begin{equation*}
\begin{aligned}
Z_2^{n,m}(\lambda,k)&=a^2(k)\bbbone\left(\frac{\lambda-\tilde{U}(k)}{\tilde{U}(k+\frac{1}{2m})-\tilde{U}(k)}- 2m((\tilde{U})^{-1}(\lambda)-k)\right)\\
&=a^2(k)\bbbone \ 2m(\lambda-\tilde{U}(k))\left(\frac{\frac{1}{2m}}{\tilde{U}(k+\frac{1}{2m})-\tilde{U}(k)}- \frac{(\tilde{U})^{-1}(\lambda)-k}{\lambda-\tilde{U}(k)}\right)\\
&=a^2(k)\bbbone \ 2m(\lambda-\tilde{U}(k))\left(\frac{\frac{1}{2m}}{\tilde{U}(k+\frac{1}{2m})-\tilde{U}(k)}- \frac{(\tilde{U})^{-1}(\lambda)-k}{\tilde{U}((\tilde{U})^{-1}(\lambda))-U(k)}\right).\\
\end{aligned}
\end{equation*}
Hence, since $\tilde{U}\in \mathcal{C}^1_{loc}(\mathbb{R})$, we see that for all $\lambda \in (\tilde{U}(k),\tilde{U}(k+\frac{1}{2m}))$ and  for $m\to \infty$
\begin{equation}\label{domu}
\begin{aligned}
|Z_2^{n,m}(\lambda,k)|&\le |a^2(k)|\left|\frac{\frac{1}{2m}}{\tilde{U}(k+\frac{1}{2m})-\tilde{U}(k)}- \frac{(\tilde{U})^{-1}(\lambda)-k}{\tilde{U}((\tilde{U})^{-1}(\lambda))-U(k)}\right|\\
&\to |a^2(k)|\left|\frac{1}{\tilde{U}'(k)}- \frac{1}{\tilde{U}'(k)}\right|=0.
\end{aligned}
\end{equation}
Consequently, using the uniform convergence results \eqref{Zunif} and \eqref{domu}, we let  $m\to \infty$ in \eqref{Kruzp4} to obtain that in the sense of distributions
\begin{equation*}
\begin{aligned}
&\partial_t\langle(U^1(\lambda)-U^2(k))_+,\nu^{1}_{(t,x)}(\lambda)\otimes\nu^{2}_{(t,x)}(k)\rangle \\
&\quad+ \diver_x \langle (\bA^1(\lambda)-\bA^2(k))\bbbone_{\{\lambda \ge \tilde{U}(k)\}},\nu^{1}_{(t,x)}(\lambda)\otimes\nu^{2}_{(t,x)}(k)\rangle\\
& \le \langle (a^{1}(\lambda)-a^2(k))\bbbone_{\{\lambda \ge \tilde{U}(k)\}},\nu^{1}_{(t,x)}(\lambda)\otimes \nu^2_{(t,x)}(k)\rangle.
\end{aligned}
\end{equation*}
Thus, repeating the same procedure from the previous step, we obtain \eqref{resultin2}.
\end{proof}

\subsection{Kinetic approximation}
In this subsection, we will try to follow the standard kinetic approximation scheme, which in our case takes the form
\begin{align}
\left.\begin{aligned}
U'(\xi)\partial_t f_{\lambda} + \bA'(\xi)\cdot \nabla_x f_{\lambda} + a(\xi)\partial_{\xi} f_{\lambda}&=\lambda (\chi(\xi,v_{\lambda})-f_{\lambda})\textrm{ in }\mathbb{R}_+^{d+2},\\
v_{\lambda}(t,x)&:=\int_{\mathbb{R}} f_{\lambda}(t,x,\xi)\dxi,\\
f_{\lambda}(0,x,\xi)&=\chi(\xi,v_0(x)).%,\\
%\chi(\cdot,u_0)&=\chi(0) &&\textrm{in }\mathbb{R}^{d+1},
\end{aligned}\right.
\label{E1ap}
\end{align}
However, due to the possible non-smoothness of the parametrization of the graphs $\mathcal{A}_{\bF}$ and $\mathcal{A}_G$ the above problem might not be easily solvable. Therefore, we formulate in this subsection all  results for smooth parametrization and provide uniform estimates that will be however independent of the regularization. This result then allows us to pass to the limit from the kinetic approximative scheme to the weak entropy solution of the problem in Section~\ref{S4}. The basis of the forthcoming existence  analysis it the following lemma.
\begin{Lemma}\label{KAex}
Let $U\in \mathcal{C}^{\infty}(\R)$, $\bA\in \mathcal{C}^{\infty}(\mathbb{R}; \mathbb{R}^d)$ and $a\in \mathcal{C}^{\infty}(\mathbb{R})$ be globally Lipschitz, i.e., there exists a constant $K$ such that for all $\xi \in \R$
$$
|U'(\xi)|+ |a'(\xi)|+ |\bA'(\xi)|\le K.
$$
Further, assume that $U(0)=a(0)=0$ and that $U$ is uniformly increasing, i.e., there exists $\varepsilon>0$ such that for all $\xi\in \R$ there holds
\begin{equation}\label{strictinc}
 U'(\xi)\ge \varepsilon.
\end{equation}
Then for any $v_0 \in L^1(\R^{d})\cap L^{\infty}(\R^{d})$ and all $\lambda >0$ there exists unique bounded $f_{\lambda}\in \mathcal{C}(0,T; L^1(\R^{d+1}))$ solving for all $\psi \in \mathcal{D}((0,T)\times \R^{d+1})$
\begin{align}
\left.\begin{aligned}
&\int_{\R^{d+2}}U'(\xi)f_{\lambda}(t,x,\xi)\partial_t \psi(t,x,\xi) + f_{\lambda}(t,x,\xi)\bA'(\xi)\cdot \nabla_x \psi(t,x,\xi)  \dxi \dx \dt\\
&+\int_{\R^{d+2}} a'(\xi)f_{\lambda}(t,x,\xi)\psi(t,x,\xi) + a(\xi) f_{\lambda}(t,x,\xi)\partial_{\xi}\psi(t,x,\xi)\dxi \dx \dt \\
&=\int_{\R^{d+2}}\lambda (f_{\lambda}(t,x,\xi)-\chi(\xi,v_{\lambda}(t,x)))\psi(t,x,\xi) \dxi \dx \dt,\\
&v_{\lambda}(t,x)=\int_{\mathbb{R}} f_{\lambda}(t,x,\xi)\dxi \qquad \textrm{ and } \qquad  f_{\lambda}(0,x,\xi)=\chi(\xi,v_0(x)).%,\\
%\chi(\cdot,u_0)&=\chi(0) &&\textrm{in }\mathbb{R}^{d+1},
\end{aligned}\right.
\label{E1ap-weak}
\end{align}
Moreover, if $U$ and $a$ satisfy \eqref{Ag_mon} and \eqref{Ag_bound} then the following estimates hold
\begin{equation}\label{uniffl}
\begin{split}
&\sup_{t\in (0,T)} \int_{\R^{d+1}}U'(\xi)|f_{\lambda}(t,x,\xi)|\dxi \dx  \le C(L,T)\int_{\R^{d}}|U(v_0(x))|\dx, \\
&\sup_{t\in (0,T)} \|f_{\lambda}(t)\|_{\infty} \le 1,\\
&\supp f_{\lambda} \subset \left\{(t,x,\xi)\in [0,T]\times \R^{d+1}: \, |\xi|\le e^{Lt}\max (L,\|v_0\|_{\infty})\right\}.\\
&\sup_{t\in (0,T)}\|v_{\lambda}(t)\|_{\infty} \le e^{LT}\max (L,\|v_0\|_{\infty})
\end{split}
\end{equation}
with a constant $C$ independent of $\varepsilon$ and $K$. In addition, we have the following sign property
\begin{equation}
|f_{\lambda}(t,x,\xi)|= f_{\lambda}(t,x,\xi)\sgn\xi \qquad \textrm{ a.e. in } (0,T)\times \R^{d+1}.\label{signf}
\end{equation}
\end{Lemma}
\begin{proof}
We follow very closely the proof in \cite{Perthame02} and therefore we frequently omit technical details and focus only on the key differences caused by the presence of $U$ (and uniform estimates when $\varepsilon \to 0_+$) and $a$ (and its possible non-smoothness).

First, due to the smoothness of $\bA$, $a$ and $U$ and thanks to the property \eqref{strictinc}, we know that for an arbitrary $g\in \mathcal{D}(\R^{d+2})$ and $f^0 \in \mathcal{D}(\R^{d+1})$ we can find a $\mathcal{C}^1$ solution $f_g$ to the following system (see \cite[Section 3.2]{Ev2010} for details).
\begin{equation}\label{clasfg}
\begin{aligned}
U'(\xi)\partial_t f_g + \bA'(\xi)\cdot \nabla_x f_g + a(\xi)\partial_{\xi} f_g + \lambda f_g &= \lambda g &&\textrm{in } (0,T)\times \R^{d+1},\\
f_g(0)&=f^0 &&\textrm{in }  \R^{d+1}.
\end{aligned}
\end{equation}
In addition, for all $t\in [0,T]$ the function $f_g(t)$ is compactly supported in $\R^{d+1}$. Next, we strengthen this result and show that the above equation is uniquely solvable even for any $g\in L^1(0,T; L^1(\R^{d+1}))$ and $f^0\in L^1(\R^{d+1})$. In what follows, we focus on the a~priori estimates for $\mathcal{C}^1$ solutions. Multiplying \eqref{clasfg} by $p|f_{g}|^{p-2}f_g$ and using the Young inequality, we get
\begin{equation}\label{est1}
\begin{split}
\partial_t \left(U'(\xi) |f_g|^p\right) + \bA'(\xi)\cdot \nabla_x |f_g|^p + a(\xi)\partial_{\xi} |f_g|^p + \lambda |f_g|^p &\le  \lambda|g|^p.
\end{split}
\end{equation}
Next, integrating the result with respect to $x$ and $\xi$, using the integration by parts and the fact that $f_g$ has compact support, we deduce
\begin{equation}\label{est2}
\begin{split}
\frac{d}{dt}\int_{\R^{d+1}}U'(\xi) |f_g|^p \dxi \dx \le \int_{\R^{d+1}} a'(\xi)|f_g|^p +  \lambda (|g|^p-|f_g|^p)\dx \dxi.
\end{split}
\end{equation}
Thus, using \eqref{strictinc} and the Grownwall inequality, we get
\begin{equation}\label{apr45}
\sup_{t\in (0,T)} \|f_{g}(t)\|_p \le C(K,\varepsilon,\lambda)\left(\|f^0\|_p + \left(\int_0^T \|g\|_p^p \dt \right)^{\frac{1}{p}}\right).
\end{equation}
Consequently, using the linearity of \eqref{clasfg}, the above estimate and the density of smooth functions, we see that for any $g\in L^1(0,T;L^{1}(\R^{d+1}))$ and $f^0\in L^1 (\R^{d+1})$ we can find a weak solution $f_g$ to \eqref{clasfg}, i.e., $f\in \mathcal{C}([0,T]; L^1(\R^{d+1}))$ solving for any $\psi \in \mathcal{D}((-\infty,T)\times \R^{d+1})$
\begin{equation}\label{weakfg}
\begin{split}
\int_0^T \int_{\R^{d+1}} \!\!\! U'(\xi) f_g \partial_t \psi + \bA'(\xi)f_g\cdot \nabla_x \psi + \partial_{\xi}\left(a(\xi)\psi\right) f_g - \lambda  \psi(f_g- g)\dxi \dx \dt\\
=-\int_{\R^{d+1}}U'(\xi) f^0(x,\xi) \psi(0,x,\xi)\dxi \dx.
\end{split}
\end{equation}
In addition, if $g$ and $f^0$ are bounded then $f_g$ is bounded as well. Moreover, since $U$, $\bA$ and $a$ are smooth, we see that any such solution is also renormalized solution, i.e., for any $h\in \mathcal{C}^1(\R)$ solving  in the sense of distribution in $(0,T)\times \R^{d+1}$
\begin{equation}\label{clasfgren}
\begin{split}
U'(\xi)\partial_t h(f_g) + \bA'(\xi)\cdot \nabla_x h(f_g) + a(\xi)\partial_{\xi} h(f_g) + \lambda h'(f_g)f_g &= \lambda h'(f_g)g.
\end{split}
\end{equation}

In order to finally also show \eqref{signf}, consider bounded compactly supported smooth $g$ and $f^0$ fulfilling for all $\xi \neq 0$ and all $(t,x)\in (0,T)\times \R^{d}$
\begin{equation}\label{signhlp}
\sgn f^0(x,\xi)=\sgn g(t,x,\xi) = \sgn \xi.
\end{equation}
Then, multiplying \eqref{clasfg} by $f^-_g:=\min(0,f_g)$ (or using $h(s):= \min(0,s)$ in \eqref{clasfgren}), we get for all $(t,x)\in (0,T)\times \R^d$ and all $\xi >0$
\begin{equation}\label{clasfgsgn2}
\begin{split}
U'(\xi)\partial_t |f^-_g|^2 + \bA'(\xi)\cdot \nabla_x |f^-_g|^2 + a(\xi)\partial_{\xi} |f^-_g|^2 + 2\lambda |f^-_g|^2 &= 2\lambda g f_g^-\le 0.
\end{split}
\end{equation}
Integrating the result over $x\in \R^{d}$ and $\xi\in (0,\infty)$, using integration by parts (note that all boundary terms vanish since $f_{g}$ has compact support and $a(0)=0$) we get that
\begin{equation}\label{clasfgsgn}
\begin{split}
\frac{d}{dt} \int_{\R^d}\int_0^{\infty}U'(\xi) |f^-_g|^2\dxi\dx \le \int_{\R^d}\int_0^{\infty} a'(\xi)|f^-_g|^2 \dxi \dx.
\end{split}
\end{equation}
Hence, using \eqref{strictinc}, \eqref{signhlp}, smoothness of $a$ and the Gronwall lemma, we see that
$$
f^-_g =0 \implies f_g\ge 0 \quad \textrm {on } (0,T)\times \R^d \times (0,\infty).
$$
Similarly, one can deduce the same result also for negative $\xi$ and consequently get that
\begin{equation}\label{signhlp2}
\sgn f_g(t,x,\xi)=\sgn \xi \qquad \textrm{almost everywhere in } (0,T) \times \mathbb{R}^{d+1}.
\end{equation}

Next, we focus on solvability of \eqref{E1ap-weak}. Defining $f^0(x,\xi):=\chi(\xi,v_0(x))$, we see that $f^0$ is bounded compactly supported  and in addition satisfies
\begin{equation}\label{we1}
\int_{\R^{d+1}} |f^0(x,\xi)|\dxi\dx =\int_{\R^{d+1}} |\chi(\xi,v_0(x))|\dxi\dx \overset{\eqref{prop2}}= \int_{\R^d}|v_0(x)|\dx <\infty.
\end{equation}
Similarly, setting for arbitrary $\bar{f} \in L^1((0,T)\times \R^{d+1})$
\begin{equation}\label{dosad}
\bar{v}(t,x):= \int_{\R}\bar{f}(t,x,\xi)\dxi, \qquad g(t,x,\xi):=\chi(\xi,\bar{v}(t,x)),
\end{equation}
we see that $g$ is bounded and satisfies
\begin{equation}\label{we2}
\begin{split}
\int_0^T\int_{\R^{d+1}} |g(t,x,\xi)|\dxi\dx \dt&=\int_0^T\int_{\R^{d+1}} |\chi(\xi,\bar{v}(t,x))|\dxi\dx\dt\\
 &\overset{\eqref{prop2}}= \int_0^T\int_{\R^d}|\bar{v}(t,x)|\dx\dt \\
 &\le \int_0^T \int_{\R^{d+1}} |\bar{f}(t,x,\xi)|\dxi\dx\dt <\infty.
\end{split}
\end{equation}
In addition, we see that for such a choice of $f^0$ and $g$, the relation \eqref{signhlp} trivially holds. Therefore, for any $\bar{f}$ we can construct a unique solution to \eqref{weakfg}, which in addition satisfies \eqref{clasfgren} in the distributional sense. Finally, we show that the mapping $\bar{f}\mapsto f_g$ is a contraction in $L^1(0,T^*;L^1(\R^{d+1}))$ for some $T^*>0$ and therefore has a fixed point. In addition, from the uniform estimates below, we shall see that this local in time solution can be extended to the whole time interval $(0,T)$.

To show the contraction property, assume that $\bar{f}_1,\bar{f}_2\in L^1(0,T^*; L^1(\R^{d+1}))$,  consider two corresponding solutions $f_1:=f_{g_1}$ and $f_2:=f_{g_2}$ and denote $w:=f_1-f_2$ and $\bar{w}:=\bar{f}_1-\bar{f}_2$. Then, it  satisfies in the sense of distributions
\begin{equation}\label{deiclasfg}
\begin{split}
U'(\xi)\partial_t w + \bA'(\xi)\cdot \nabla_x w + a(\xi)\partial_{\xi} w + \lambda w &= \lambda (\chi(\xi,\bar{v}_1)-\chi(\xi,\bar{v}_2)),\\
w(0,x,\xi)&=0.
\end{split}
\end{equation}
Due to the smoothness of the coefficients, this equation can be also renormalized and therefore we even get
\begin{equation}\label{deiclasfgm}
\begin{split}
U'(\xi)\partial_t |w| + \bA'(\xi)\cdot \nabla_x |w| + a(\xi)\partial_{\xi} |w| + \lambda |w| &\le \lambda |\chi(\xi,\bar{v}_1)-\chi(\xi,\bar{v}_2)|.
\end{split}
\end{equation}
Since $w\in L^1$, we can in fact integrate the above inequality over $\R^{d+1}$ to obtain\footnote{This step can be rigorously justified by testing  with $\psi\in \mathcal{D}(\R^{d+1})$ and letting $\psi\nearrow 1$ since all terms involving $\nabla \psi$ vanish.}
\begin{equation}\label{deiclasfgm2}
\begin{split}
&\frac{d}{dt}\int_{\R^{d+1}} U'(\xi)|w| \dxi \dx  + \lambda \int_{\R^{d+1}}|w|\dxi\dx\\
 &\qquad \le \lambda \int_{\R^{d+1}}|\chi(\xi,\bar{v}_1)-\chi(\xi,\bar{v}_2)|\dxi \dx +\int_{\R^{d+1}}a'(\xi)|w|\dxi\dx\\
&\qquad = \lambda \int_{\R^{d}}|\bar{v}_1-\bar{v}_2|\dx +\int_{\R^{d+1}}a'(\xi)|w|\dxi\dx\\
&\qquad \le \lambda \int_{\R^{d+1}}|\bar{w}|\dxi \dx +\int_{\R^{d+1}}a'(\xi)|w|\dxi\dx,
\end{split}
\end{equation}
where we used \eqref{prop3} to obtain the second equality and the triangle inequality to get the last inequality. Consequently, using our assumptions on $U$ and $a$ this inequality reduces to (note that $C>1$)
\begin{equation*}
\begin{split}
\frac{d}{dt}\int_{\R^{d+1}} U'(\xi)|w| \dxi \dx  \le C(K,\varepsilon,\lambda)\int_{\R^{d+1}} U'(\xi)(|w| +|\bar{w}|)\dxi \dx.
\end{split}
\end{equation*}
This  finally leads after using the Gronwall lemma and the fact that $w(0)=0$ to
\begin{equation*}
\begin{split}
\int_{\R^{d+1}} U'(\xi)|w(t)| \dxi \dx  \le e^{C(K,\varepsilon,\lambda)t}\int_0^t\int_{\R^{d+1}}e^{-C(K,\varepsilon,\lambda)s}U'(\xi)|\bar{w}|\dxi \dx \ds.
\end{split}
\end{equation*}
Integration of the result over $(0,T^*)$ and using the integration by parts then gives us the following estimate
\begin{equation*}
\begin{split}
\int_0^{T^*}\int_{\R^{d+1}} U'(\xi)|w(t)| \dxi \dx\dt  &\le \int_0^{T^*}\int_0^t\int_{\R^{d+1}}e^{C(K,\varepsilon,\lambda)(t-s)}U'(\xi)|\bar{w}(s)|\dxi \dx \ds\dt\\
&\le \frac{e^{C(K,\varepsilon,\lambda)T^*}}{C(K,\varepsilon,\lambda)}\int_0^{T^*}\int_{\R^{d+1}}U'(\xi)|\bar{w}(t)|\dxi \dx \dt.
\end{split}
\end{equation*}
Finally, since $C>1$ we can find $T^*$ so small that for some $\gamma<1$
\begin{equation*}
\begin{split}
\int_0^{T^*}\int_{\R^{d+1}} U'(\xi)|w| \dxi \dx\dt
&\le \gamma\int_0^{T^*}\int_{\R^{d+1}}U'(\xi)|\bar{w}|\dxi \dx \dt
\end{split}
\end{equation*}
and thus we got a contraction on the space $L^{1}((0,T^*)\times \R^d; L^1(\R;U'\dxi))$, which is thanks to \eqref{strictinc} identical to $L^1$. Consequently, we get a fix point, i.e., the solution to \eqref{E1ap-weak}. In addition, it is evident that $f_{\lambda}$ satisfies \eqref{signf}. Moreover, it is also the renormalized solution, i.e., it satisfies  for all  Lipschitz $h$ in the sense of distributions
\begin{equation}\label{ren4}
\begin{split}
U'(\xi)\partial_t h(f_{\lambda}) + \bA'(\xi)\cdot \nabla_x h(f_g) + a(\xi)\partial_{\xi} h(f_{\lambda}) + \lambda h'(f_{\lambda})f_{\lambda} &= \lambda h'(f_{\lambda})\chi(\xi,v_{\lambda}).
\end{split}
\end{equation}

The rest of the proof is devoted to the uniform ($\varepsilon$ and $K$ independent) estimates. Setting $h(s):=|s|$ in \eqref{ren4} and integrating over $\mathbb{R}^{d+1}$, we get by using integration by parts (note here that $f_{\lambda}(t)\in L^1(\R^{d+1})$ and that $\sgn f_{\lambda}=\sgn \chi$)
\begin{equation*}
\begin{split}
&\frac{d}{dt}\int_{\R^{d+1}} U'(\xi)|f_{\lambda}| \dxi \dx  + \lambda \int_{\R^{d+1}}|f_{\lambda}|\dxi\dx\\
&\qquad = \lambda \int_{\R^{d+1}}|\chi(\xi,v_{\lambda})|\dxi \dx +\int_{\R^{d+1}}a'(\xi)|f_{\lambda}|\dxi\dx\\
&\qquad \overset{\eqref{prop3}}= \lambda \int_{\R^{d}}|v_{\lambda}| \dx +\int_{\R^{d+1}}a'(\xi)|f_{\lambda}|\dxi\dx\\
 &\qquad \le \lambda \int_{\R^{d+1}}|f_{\lambda}|\dxi \dx +\int_{\R^{d+1}}a'(\xi)|f_{\lambda}|\dxi\dx,
\end{split}
\end{equation*}
which directly leads to
\begin{equation*}
\begin{split}
&\frac{d}{dt}\int_{\R^{d+1}} U'(\xi)|f_{\lambda}| \dxi \dx \le \int_{\R^{d+1}}a'(\xi)|f_{\lambda}|\dxi\dx.
\end{split}
\end{equation*}
Next, since $a$ and $U$ are assumed to be smooth, it follows from \eqref{Ag_mon} that $a'\le LU'$ and therefore
\begin{equation}\label{ineq}
\begin{split}
&\frac{d}{dt}\int_{\R^{d+1}} U'(\xi)|f_{\lambda}| \dxi \dx \le L\int_{\R^{d+1}}U'(\xi)|f_{\lambda}|\dxi\dx.
\end{split}
\end{equation}
The Gronwall lemma, the identity \eqref{prop2} and the fact that $U(0)=0$ then lead for all $t\in (0,T)$ to the estimate
\begin{equation*}
\begin{split}
&\int_{\R^{d+1}}U'(\xi)|f_{\lambda}(t,x,\xi)| \dxi \dx \le e^{Lt}\int_{\R^{d+1}}U'(\xi)|f^0(\xi,x)| \dxi \dx\\
&\qquad= e^{Lt}\int_{\R^{d+1}}U'(\xi)|\chi(\xi,v_0(x))| \dxi \dx = e^{Lt}\int_{\R^{d}}|U(v_0(x))|\dx,
\end{split}
\end{equation*}
which directly implies the first estimate in \eqref{uniffl}. To show the second part of the estimate, we set $h(s):=(s-1)_+$ in \eqref{ren4} and after integration over $\mathbb{R}^{d+1}$ and integration by parts (notice that all ``boundary" terms again vanish), we obtain
\begin{equation*}
\begin{split}
&\frac{d}{dt}\int_{\R^{d+1}}U'(\xi)(f_{\lambda}-1)_+ \dxi \dx\\
&\qquad \le \int_{\R^{d+1}}a'(\xi)(f_{\lambda}-1)_+\dxi \dx + \lambda \int_{\{f_{\lambda}\ge 1\}}\chi(\xi,v_{\lambda})-f_{\lambda} \dxi \dx\\
&\qquad \le L\int_{\R^{d+1}}U'(\xi)(f_{\lambda}-1)_+\dxi \dx,
\end{split}
\end{equation*}
where we used the assumption \eqref{Ag_mon} to estimate the first integral and the obvious inequality $\chi-f_{\lambda}\le 0$ on the set, where $f_{\lambda}\ge 1$. Consequently, since $f^0(x,\xi)=\chi(\xi,v_0(x))\le 1$, we can use the Gronwall lemma to conclude $f_{\lambda}\le 1$. In the same way, we can deduce that $f_{\lambda}\ge -1$ and consequently, we deduce the second part of the estimate \eqref{uniffl}.

Finally, to finish the proof, it remains to check the last inequality in \eqref{uniffl}. First, let us consider an auxiliary function $\bar{f}:=\chi(\xi,Me^{Kt})$ with some constants $M$ and $K$. Next, we show that it is a supersolution to \eqref{E1ap}, i.e., that for arbitrary nonnegative $\psi\in \mathcal{D}((0,T)\times \R^{d+1})$ there holds
\begin{equation}
\begin{split}
&-\int_{\R^{d+2}} U'(\xi)\bar{f}\partial_t \psi + \bar{f}\bA'(\xi)\cdot \nabla_x \psi +\bar{f} \partial_{\xi}(a(\xi)\psi)\dxi \dx \dt \\
&\quad \ge \lambda \int_{\R^{d+2}} \psi(\xi)\chi\left(\xi,\int_{\R}\bar{f}(t,x,s)\ds\right) -\psi(\xi)\bar{f}(t,x,\xi)  \dxi \dx \dt.
\end{split}\label{supers}
\end{equation}
Since, it trivially follows from the definition of $\bar{f}$ that
$$
\chi\left(\xi,\int_{\R}\bar{f}(t,x,s)\ds\right)=\bar{f}(t,x,\xi),
$$
the  inequality \eqref{supers} reduces to
\begin{equation*}
\begin{split}
\int_{\R^{d+2}} U'(\xi)\bar{f}\partial_t \psi + \bar{f}\bA'(\xi)\cdot \nabla_x \psi +\bar{f} \partial_{\xi}(a(\xi)\psi) \dxi\dx\dt \le 0.
\end{split}
\end{equation*}
Hence, using integration by parts, the fact that $a(0)=0$ and the definition of $\bar{f}$, we get
\begin{equation*}
\begin{split}
&\int_{\R^{d+2}} U'(\xi)\bar{f}\partial_t \psi + \bar{f}\bA'(\xi)\cdot \nabla_x \psi +\bar{f} \partial_{\xi}(a(\xi)\psi)\dxi\dx\dt \\
&=\int_{0}^{\infty}\int_{\R^d}\int_0^{Me^{Kt}} U'(\xi)\partial_t \psi(t,x,\xi)  +\partial_{\xi}(a(\xi)\psi(t,x,\xi))\dxi\dx\dt\\
&=\int_{\R^{d+1}}a(Me^{Kt})\psi(t,x,Me^{Kt})\dx \dt +\int_{0}^{\infty}\int_{\R^d}\int_M^{Me^{Kt}} U'(\xi)\partial_t \psi(t,x,\xi) \dxi\dx\dt\\
&=\int_{\R^{d+1}}a(Me^{Kt})\psi(t,x,Me^{Kt})\dx \dt \\
&\quad +\int_{0}^{\infty}\int_{\R^d}\int_0^{t}KMe^{Ks} U'(Me^{Ks})\partial_t \psi(t,x,Me^{Ks})  \ds\dx\dt\\
&=\int_{\R^{d+1}}\left(a(Me^{Kt})-KMe^{Kt} U'(Me^{Kt})\right)\psi(t,x,Me^{Kt}) \dx \dt,
\end{split}
\end{equation*}
where for the last equality we used the Fubini theorem and the integration by parts (note that $\psi$ is compactly supported with respect to $t$).
Thus, we see that for the validity of \eqref{supers}, we need to check that
$$
a(s)\le Ks U'(s) \textrm{ for all } s\in [M,\infty).
$$
Therefore, using the assumption \eqref{Ag_bound}, we see that if
\begin{equation}\label{checkok}
\min(K,M)\ge L,
\end{equation}
where $L$ appears in \eqref{Ag_bound}, then \eqref{supers} holds. Knowing that $\bar{f}$ is a supersolution and denoting $w:=\bar{f}-f_{\lambda}$ we therefore have
$$
U'(\xi)\partial_t w + \bA'(\xi) \cdot \nabla_x w + a(\xi)\partial_{\xi} w \ge \lambda (\bar{f}-\chi(\xi,v_{\lambda}) -\bar{f}+f_{\lambda}).
$$
Again, due to the smoothness of coefficients in the transport terms, this inequality can be renormalized and after a standard procedure we get that for $w_{-}:=\min(0,w)$ and arbitrary nonnegative $\varphi \in \mathcal{D}(\R^{d+1})$
\begin{equation*}
\begin{split}
&\frac{d}{dt}\int_{\R^{d+1}} U'(\xi)|w_{-}| \varphi \dxi \dx  -\int_{\R^{d+1}} |w_{-}|\bA'(\xi) \cdot \nabla_x \varphi +\partial_{\xi}(a(\xi)\varphi)|w_{-}| \dxi \dx \\
&\quad \le \lambda \int_{\R^{d+1}\cap\{(x,\xi):\, \bar{f}\le f_{\lambda}\}} \varphi \left|(\bar{f}-\chi(\xi,v_{\lambda}))_{-}\right| -\varphi |w_{-}| \dxi\dx.
\end{split}
\end{equation*}
Since $f_{\lambda}\in L^1 \cap L^{\infty}$, $\bar{f}$ is nonnegative and  compactly supported with respect to $\xi$ and independent of $x$, and $\bA$ and $a$ are globally Lipschitz,  we can let $\varphi \nearrow 1$ in the above inequality to conclude
\begin{equation}\label{agne}
\begin{split}
&\frac{d}{dt}\int_{\R^{d+1}} U'(\xi)|w_{-}| \dxi \dx  +a'(\xi)|w_{-}|\dxi \dx \\
&\quad \le \lambda \int_{\R^{d+1}} \left|(\bar{f}-\chi(\xi,v_{\lambda}))_{-}\right| - |w_{-}| \dxi\dx.
\end{split}
\end{equation}
Finally, using the definition of $\chi$, we get the following identity
$$
\begin{aligned}
\int_{\R} \left|(\bar{f}-\chi(\xi,v_{\lambda}))_{-}\right|\dxi&=\int_{\R} \left|(\chi(\xi,Me^{Kt})-\chi(\xi,v_{\lambda}))_{-}\right|\dxi \\
&= \left|(Me^{Kt}-v_{\lambda})_{-}\right|=\left|\left(\int_{\R} \bar{f}-f_{\lambda} \dxi \right)_{-}\right|\\
&\le \int_{\R} \left|\left(\bar{f}-f_{\lambda}  \right)_{-}\right|\dxi=\int_{\R} \left|w_{-}\right|\dxi.
\end{aligned}
$$
Thus, substituting it into \eqref{agne} and using the assumption \eqref{Ag_mon}, we deduce
\begin{equation}\label{agne2}
\begin{split}
&\frac{d}{dt}\int_{\R^{d+1}} U'(\xi)|w_{-}| \dxi \dx  \le L\int_{\R^{d+1}}U'(\xi)|w_{-}|\dxi \dx.
\end{split}
\end{equation}
Therefore, using the Gronwall lemma, we see that
$$
\begin{aligned}
M\ge \|v_0\|_{\infty} &\implies \bar{f}(0) \ge f_{\lambda}(0) \implies w_- (0)=0 \implies  w\ge 0 \textrm{ in }(0,T)\times \R^{d+1}\\
&\implies f_{\lambda} \le \chi(\xi,Me^{Kt}).
\end{aligned}
$$
Hence, setting finally $M:=\max (L,\|v_0\|_{\infty})$ and $K:=L$ and using the fact that $f_{\lambda} \ge 0$ for $\xi \ge 0$, we observe
$$
f_{\lambda}(t,\xi,x)=0 \textrm{ if } \xi \ge e^{Lt}\max (L,\|v_0\|_{\infty}).
$$
Since, the procedure above was completely symmetric, we can deduce the same bound also from below and finally get the third part of \eqref{uniffl}.

To complete the proof it remains to show  the last part of \eqref{uniffl}. Having the sign property stated in \eqref{signf} and the second part of \eqref{uniffl}, it is straightforward  to conclude
$$
\begin{aligned}
|v_{\lambda}(t,x)| &= \left| \int_{\R} f(t,x,\xi)\dxi \right| \le \max \left(\int_0^{\infty} f_{\lambda}(t,x,\xi)\dxi, \int_{-\infty}^0 |f(t,x,\xi)|\dxi\right)\\
&=\max \left(\int_0^{e^{Lt}\max (L,\|v_0\|_{\infty})} f_{\lambda}(t,x,\xi)\dxi, \int_{-e^{Lt}\max (L,\|v_0\|_{\infty})}^0 |f(t,x,\xi)|\dxi\right)\\
&\le e^{Lt}\max (L,\|v_0\|_{\infty})
\end{aligned}
$$
and the last part of \eqref{uniffl} easily follows. The proof is complete.
\end{proof}

\section{Proof of Theorem~\ref{TT2}} \label{S4}

Lets us consider the given parametrization $(\bA,a,U)$ fulfilling \eqref{param}. We first focus on the existence proof in case that all of the assumptions \eqref{AF_mon}--\eqref{Ag_bound} are satisfied and assume that the initial data $u_0\in L^1(\R^d)\cap L^{\infty}(\R^d)$. The case of unbounded initial data and/or the case of relaxed assumption on the parametrization will be discussed at the end of this section. We will prove the existence theorem by using the following cascade of approximations. First, we mollify the parametrization and also the initial data, add the strictly monotone term to $a$  and introduce the $\lambda$-kinetic approximation. Based on the a~priori estimates stated in Section~\ref{S3}, we then let $\lambda\to \infty$ to pass to the measure valued solution. Using  the stability inequality we then show then it is in fact the weak entropy solution. Finally, we remove the strict monotone part from the source term and show the solvability of the original pr
 oblem an
 d also its uniqueness.

\subsection{Mollification and kinetic approximation}
We denote by $\rho_r$ the standard convolution kernel of radius $r>0$. Then for given $(U,\bA,a)$ and $u_0$ and arbitrary $\lambda>1$, we define the following mollification
\begin{equation}\label{eq_app_lambda}
\begin{aligned}
u^{\lambda}_0&:= (u_0 \bbbone_{\{|x|\le \lambda\}})*\rho_{\lambda^{-1}}\in \mathcal{D}(\R^d), \\
U^{\lambda}&:=\tilde{U}^{\lambda}*\rho_{\lambda^{-\frac14}}-\tilde{U}^{\lambda}*\rho_{\lambda^{-\frac14}}(0) + \lambda^{-1} s,\\
\bA^{\lambda}&:=(\bA\bbbone_{\{|s|\le \lambda\}})*\rho_{\lambda^{-\frac14}},\\
a^{\lambda}_{\ell,m}&:=(a^{\lambda}+a_{\ell,m})*\rho_{\lambda^{-\frac14}} -  (a^{\lambda}+a_{\ell,m})*\rho_{\lambda^{-\frac14}}(0),\\
v_0^{\lambda}&:= (U^{\lambda})^{-1}(u^{\lambda}_0),
\end{aligned}
\end{equation}
where $a_{\ell,m}$ is an arbitrary continuous bounded strictly decreasing function fulfilling $a_{\ell,m}(0)=0$ and  $a^{\lambda}$ and $\tilde{U}^{\lambda}$ are given
$$
\begin{aligned}
a^{\lambda}(s)&:=\left\{ \begin{aligned}
&a(s) &&\textrm{for } |s|\le \lambda,\\
&(s-\lambda)+a(\lambda)  &&\textrm{for } s>\lambda,\\
&(s+\lambda)+ a(-\lambda) &&\textrm{for }s<-\lambda,
\end{aligned}
\right. \\
 \tilde{U}^{\lambda}(s)&:=\left\{
\begin{aligned}
&U(s) &&\textrm{for } |s|\le \lambda,\\
&(s-\lambda)+U(\lambda)  &&\textrm{for } s>\lambda,\\
&(s+\lambda)+ U(-\lambda) &&\textrm{for }s<-\lambda.
\end{aligned}
\right.
\end{aligned}
$$
With this notation, we introduce the following kinetic approximation
\begin{align}
\left.\begin{aligned}
(U^{\lambda})'(\xi)\partial_t f_{\lambda} + (\bA^{\lambda})'(\xi)\cdot \nabla_x f_{\lambda} + a^{\lambda}_{\ell,m}(\xi)\partial_{\xi} f_{\lambda}&=\lambda (\chi(\xi,v_{\lambda})-f_{\lambda})\textrm{ in }\mathbb{R}_+^{d+2},\\
v_{\lambda}(t,x)&:=\int_{\mathbb{R}} f_{\lambda}(t,x,\xi)\dxi,\\
f_{\lambda}(0,x,\xi)&=\chi(\xi,v^{\lambda}_0(x)).%,\\
%\chi(\cdot,u_0)&=\chi(0) &&\textrm{in }\mathbb{R}^{d+1},
\end{aligned}\right.
\label{E1aplambda}
\end{align}
Since all assumptions of Lemma~\ref{KAex} are satisfied, we have the existence of a solution to \eqref{E1aplambda}. In addition, due to the assumptions on $U$ and $a$ and the definition of $a^{\lambda}_{\ell,m}$ and $U^{\lambda}$ we see that the mollified $(a^{\lambda}_{\ell,m},U^{\lambda})$ also satisfy  \eqref{Ag_mon} and \eqref{Ag_bound} and therefore we can use \eqref{uniffl} to obtain
\begin{equation}\label{unifflam}
\begin{split}
&\sup_{t\in (0,T)} \int_{\R^{d+1}}(U^{\lambda})'(\xi)|f_{\lambda}(t,x,\xi)|\dxi \dx  \le C(L,T)\int_{\R^{d}}|u^{\lambda}_0(x)|\dx \le C,\\
&\sup_{t\in (0,T)} \|f_{\lambda}(t)\|_{\infty} \le 1,\\
&\supp f_{\lambda} \subset \left\{(t,x,\xi)\in [0,T]\times \R^{d+1}: \, |\xi|\le e^{Lt}C(L,\|u_0\|_{\infty})\right\},\\
&\sup_{t\in (0,T)}\|v_{\lambda}(t)\|_{\infty} \le e^{LT}\max (L,\|v^{\lambda}_0\|_{\infty})\le C(L,T,\|u_0\|_{\infty}).
\end{split}
\end{equation}
Finally, since $f_{\lambda}$ has compact support with respect to $\xi$, we can deduce from the definition of $U^{\lambda}, \bA^{\lambda}, a^{\lambda}_{\ell,m}$, by using \eqref{eq_app_lambda} and \eqref{unifflam} and the integration by parts that for all $\varphi \in \mathcal{C}_0^1((-\infty,T)\times \R^{d}; L^{\infty}(\R))$ and such that $\varphi$ is of class  $\mathcal{C}^1$ on $(0,\infty)$ and $(-\infty,0)$ but its derivative may not exist in zero there holds (note here that for the integration by parts we use the fact that $a^{\lambda}_{\ell,m}(0)=0$)
\begin{equation}\label{nula}
\begin{split}
&\int_0^T\int_{\R^{d+1}}(\chi(\xi,v_{\lambda})-f_{\lambda})\varphi \dxi \dx \dt\\
& \le \frac{1}{\lambda} \int_0^T\int_{\R^{d+1}} \int(U^{\lambda})'(\xi)\partial_t f_{\lambda}\varphi + (\bA^{\lambda})'(\xi)\cdot \nabla_x f_{\lambda}\varphi + a^{\lambda}_{\ell,m}(\xi)\partial_{\xi} f_{\lambda}\varphi \dxi \dx \dt\\
& =-\frac{1}{\lambda}\int_0^T \int_{\R^{d+1}} \int(U^{\lambda})'(\xi) f_{\lambda}\partial_t\varphi + (\bA^{\lambda})'(\xi)  f_{\lambda}\cdot \nabla_x\varphi + a^{\lambda}_{\ell,m}(\xi) f_{\lambda}\partial_{\xi} \varphi\\
&\qquad {} \qquad +(a^{\lambda}_{\ell,m})'(\xi) f_{\lambda} \varphi \dxi \dx \dt + \frac{1}{\lambda} \int_{\R^{d}} (U^{\lambda})'(\xi)\varphi(0,x,\xi)\chi(\xi,v_{\lambda}) \\
&\le  \frac{C(T,L, U,\bA, a,\|u_0\|_{\infty})}{\lambda^{\frac34}} \int_{\supp f_{\lambda}} 1+ |\partial_t\varphi| + |\nabla_x\varphi| + |\partial_{\xi} \varphi|+ |\varphi| \dxi \dx \dt\\
&\le \frac{C(T,L, U,\bA, a,\|u_0\|_{\infty})}{\lambda^{\frac34}} \left(1+\int_{Q_K}\left(\int_{-K}^0 |\nabla \varphi|\dxi \right) +\left(\int_0^{K}|\nabla \varphi|\dxi \right) \dx \dt\right),
\end{split}
\end{equation}
where $K$ is sufficiently large depending on $\|v_0\|_{\infty}$ and $T$, $Q_K$ is a ball of radius $K$ in $\mathbb{R}^{d+1}_+$ and $\nabla \varphi$ indicates the derivatives with respect to $t,x,\xi$. Note here that since we do not assume that $\partial_{\xi}\varphi$ exists on $\R$ (but exists on $\R_+$ and $\R_{-}$) we decomposed the last integral onto the corresponding parts.

\subsection{Limit $\lambda \to \infty$}

In this subsection, we let $\lambda \to \infty$ in \eqref{E1aplambda}. First, we show that the weak limit of $v_{\lambda}$ is an entropy measure valued solution and then with the help of the stability inequality, we show that it is in fact the weak entropy solution.

Due to the a~priori estimates \eqref{E1aplambda}--\eqref{nula} and since $\chi$ is bounded function,  we can extract a non-relabeled subsequence such that
\begin{align}
\chi(\eta,v_{\lambda})-f_{\lambda} &\rightharpoonup^*  0 &&\textrm{ weakly$*$ in }  L^{\infty}((0,T)\times \R^{d+1}),\label{c1}\\
v_{\lambda} &\rightharpoonup^* v &&\textrm{ weakly$^*$ in } L^{\infty}(0,T; L^{\infty}(\R^d)).\label{c2}
\end{align}
Next, we denote by $\nu_{(t,x)}$ the Young measure corresponding to $v$, i.e., we have that $\nu \in L^{\infty}((0,T)\times \R^d; \mathcal{P}_0(\R))$, where $\mathcal{P}_0(\R)$ denotes the space of probability compactly supported measures, and for all $g\in \mathcal{C}(\R)$ we have that
\begin{equation}
g(v_{\lambda}) \rightharpoonup^* \overline{g}= \int_{\R} g(\xi) \mathrm{d}\nu_{(t,x)}(\xi)=:\langle g,\nu_{(t,x)}\rangle, \label{c3}
\end{equation}
where the weak star limit is understood in $L^{\infty}((0,T)\times \R^d)$. Consequently, it follows from \eqref{c1}--\eqref{c2} that for all $g\in \mathcal{C}^1_{loc}(\R)$ we have that
\begin{equation}\label{IDf}
\begin{aligned}
\int_{\R} f_{\lambda}(t,x,\xi)g'(\xi)\dxi &\rightharpoonup^* \lim_{\lambda\to \infty} \int_{\R}\chi(\eta,v_{\lambda}(t,x))g'(\xi)\dxi \\
&= \lim_{\lambda\to \infty} g(v_{\lambda})-g(0)=\langle g(\xi),\nu_{(t,x)}(\xi)\rangle -g(0),
\end{aligned}
\end{equation}
where all limits are understood as weak star limits in $L^{\infty}((0,T)\times \R^d)$. In addition, we can strengthen the relation \eqref{IDf} in the following way. Assume that $\{g^{\lambda}\}_{\lambda}\subset \mathcal{C}^2(-\infty,0)\cup \mathcal{C}^2(0,\infty)$ are such that
\begin{equation}
\begin{aligned}
g^{\lambda} &\to g &&\textrm{strongly in }\mathcal{C}_{loc}(\R),\\
(|(g^{\lambda})'|+|(g^{\lambda})''|)\lambda^{-\frac34} &\to 0 &&\textrm{strongly in } \mathcal{C}_{loc}(-\infty,0]\cap \mathcal{C}_{loc}[0,\infty).
\end{aligned}\label{forseq}
\end{equation}
Then using the estimate \eqref{nula}, the uniform convergence of $g^{\lambda}$ and the relation \eqref{IDf}, we deduce that
$$
\begin{aligned}
&\lim_{\lambda \to \infty} \int_{\R^{d+2}} (g^{\lambda})'(\xi) f_{\lambda}(t,x,\xi) \varphi(t,x) \dxi \dx \dt\\
&=\lim_{\lambda \to \infty} \int_{\R^{d+2}} (g^{\lambda})'(\xi) (f_{\lambda}(t,x,\xi)-\chi(\xi,v_{\lambda}(t,x))) \varphi(t,x) \dxi \dx \dt\\
&\qquad +\lim_{\lambda \to \infty} \int_{\R^{d+2}}(g^{\lambda})'(\xi) \chi(\xi,v_{\lambda}(t,x)) \varphi(t,x) \dxi \dx \dt\\
&\le \lim_{\lambda \to \infty} C(\varphi)\lambda^{-\frac34}(\|g^{\lambda}\|_{\mathcal{C}^2_{loc}(-\infty,0]}+\|g^{\lambda}\|_{\mathcal{C}^2_{loc}[0,\infty)})\\
&\qquad +\lim_{\lambda \to \infty} \int_{\R^{d+1}}(g^{\lambda}(v_{\lambda}(t,x))-g^{\lambda}(0))\varphi(t,x)  \dx \dt\\
&= \int_{\R^{d+1}} \left(\langle g(\xi),\nu_{(t,x)}(\xi)\rangle -g(0)\right) \varphi(t,x)\dx \dt.
\end{aligned}
$$
Consequently, considering also $-\varphi$ we can deduce that in the sense of distributions
\begin{equation}
\int_{\R} (g^{\lambda})'(\xi) f_{\lambda}(t,x,\xi)  \dxi \rightharpoonup  \langle g(\xi),\nu_{(t,x)}(\xi)\rangle -g(0) \qquad \textrm{ in } \mathcal{D}'([0,T)\times \R^d). \label{finalklo}
\end{equation}

With the help of this identification of the limit, we can let $\lambda \to \infty$ in \eqref{E1aplambda} to obtain the measure valued entropy solution. Indeed, let $E\in \mathcal{C}^2(\R)$ be arbitrary convex function such that $E''\in \mathcal{D}(\R)$ and let $\varphi \in \mathcal{D}((-\infty,T)\times \R^{d})$ be arbitrary nonnegative function. If we multiply \eqref{E1aplambda} by $\varphi(t,x)E'(\xi)$ and integrate the result over $(0,T)\times \R^{d+1}$ and use integration by parts we obtain that
\begin{align}
\left.\begin{aligned}
\lambda &\int_0^T \int_{\R^{d+1}}E'(\xi) (\chi(\xi,v_{\lambda})-f_{\lambda})\varphi \dxi \dx \dt \\ &=\int_0^T\int_{\R^{d+1}}\left((U^{\lambda})'(\xi)E'(\xi)\partial_t f_{\lambda} + (\bA^{\lambda})'(\xi)E'(\xi)\cdot \nabla_x f_{\lambda}\right.\\
 &\qquad {} \qquad \left. + a^{\lambda}_{\ell,m}(\xi)E'(\xi)\partial_{\xi} f_{\lambda}\right)\varphi \dxi\dx\dt\\
&=-\int_0^T\int_{\R^{d+1}}(Q_U^{\lambda})'(\xi)f_{\lambda}\partial_t \varphi + (Q_{\bA}^{\lambda})'(\xi)f_{\lambda}\cdot \nabla_x \varphi +(a^{\lambda}_{\ell,m}(\xi)E'(\xi))' f_{\lambda}\varphi\dxi\dx\dt\\
&\qquad +\int_{\R^{d+1}}Q_U^{\lambda}(v_0^{\lambda}(x))\varphi(0,x) \dx,
\end{aligned}\right.
\label{E1aplambda2}
\end{align}
where
$$
\begin{aligned}
Q_U^{\lambda}(s)&:=\int_0^s (U^{\lambda})'E'\dxi = U^{\lambda}(s)E'(s)-\int_0^s U^{\lambda}E''\dxi\\
Q_{\bA}^{\lambda}(s)&:=\int_0^s (\bA^{\lambda})' E'\dxi = (\bA(s)^{\lambda}(s)E'(s)-\bA(s)^{\lambda}(0)E'(0))-\int_0^s \bA^{\lambda}E''\dxi.
\end{aligned}
$$
Consequently, due to the definition of the mollified quantities, we see that $Q_U^{\lambda}$, $Q_{\bA}^{\lambda}$ and $a^{\lambda}_{\ell,m}E'$ can be used as $g^{\lambda}$ in \eqref{forseq} and \eqref{finalklo} together with \eqref{E1aplambda2} implies that
\begin{align}
\left.\begin{aligned}
&\lim_{\lambda\to \infty} \lambda \int_0^T \int_{\R^{d+1}}E'(\xi) (\chi(\xi,v_{\lambda})-f_{\lambda})\varphi \dxi \dx \dt \\
&=-\int_0^T\int_{\R^{d}}\langle (Q_U(\xi),\nu_{(t,x)}(\xi)\rangle \partial_t \varphi(t,x) + \langle Q_{\bA}(\xi), \nu_{(t,x)}(\xi)\rangle \cdot \nabla_x \varphi(t,x) \dx \dt\\
&\qquad -\int_0^T \int_{\R^d} \langle (a(\xi)+a_{\ell,m}(\xi))E'(\xi), \nu_{(t,x)}(\xi)\rangle \varphi(t,x)\dx\dt \\
&\qquad +\int_{\R^{d+1}}Q_U(v_0(x))\varphi(0,x) \dx.
\end{aligned}\right.
\label{E1aplambda3}
\end{align}
We finally, focus on the term on the left hand side. Using the fact that $f_{\lambda}$ has compact support with respect to $\xi$ and $v_{\lambda}$ is bounded, we can use the integration by parts formula to obtain (note that $\int_{-\infty}^{\infty}f_{\lambda}(t,x,s)-\chi(s,v_{\lambda}(t,x))\ds=0$)
$$
\begin{aligned}
&\int_0^T \int_{\R^{d+1}}E'(\xi) (\chi(\xi,v_{\lambda})-f_{\lambda})\varphi \dxi \dx \dt \\
&= -\int_0^T \int_{\R^d} \int_{\R} E'(\xi)\left(\frac{d}{\dxi}\int_{-\infty}^{\xi}f_{\lambda}(t,x,s)-\chi(s,v_{\lambda}(t,x))\ds \right)\varphi(t,x)\dxi\dx\dt\\
&= \int_0^T \int_{\R^d} \int_{\R} E''(\xi)\int_{-\infty}^{\xi}f_{\lambda}(t,x,s)-\chi(s,v_{\lambda}(t,x))\ds \varphi(t,x)\dxi\dx\dt.
\end{aligned}
$$
Then, since  $|f_{\lambda}|\le 1$ and $\sgn f_{\lambda}(s)=\sgn s$ we can deduce, see \cite[Theorem 2.1.1]{Perthame02} for details, that for all $\xi\in \R$
$$
\int_{-\infty}^{\xi}f_{\lambda}(s,x,t)-\chi(s,v_{\lambda}(x,t))\ds\le 0.
$$
Hence it directly follows that for convex  $E$  and nonnegative $\varphi$, the relation \eqref{E1aplambda3} can be rewritten as
\begin{align}
\left.\begin{aligned}
&-\int_0^T\int_{\R^{d}}\langle (Q_U(\xi),\nu_{(t,x)}(\xi)\rangle \partial_t \varphi(t,x) + \langle Q_{\bA}(\xi), \nu_{(t,x)}(\xi)\rangle \cdot \nabla_x \varphi(t,x) \dx \dt\\
&\qquad -\int_0^T \int_{\R^d} \langle (a(\xi)+a_{\ell,m}(\xi))E'(\xi), \nu_{(t,x)}(\xi)\rangle \varphi(t,x)\dx\dt \\
&\qquad +\int_{\R^{d+1}}Q_U(v_0(x))\varphi(0,x) \dx\le 0
\end{aligned}\right.
\label{E1aplambda4}
\end{align}
and consequently $\nu_{(t,x)}$ is an entropy measure valued solution for $(U,\bA, a+a_{\ell,m},v_0)$ with some $v_0 \in L^{\infty}(\R^d)$ satisfying $U(v_0)=u_0$ in $\R^d$.

Our goal now is to use the stability inequality. However, for doing so, we first need to show that $\langle |U|,\nu\rangle \in L^{\infty}(0,T; L^1(\R^d))$. For this purpose assume that $\varphi \in \mathcal{D}((0,T)\times \R^d)$ is an arbitrary nonnegative function. Having in mind that $|(U^{\lambda})''|\le C\lambda^{\frac12}$, we can now use \eqref{finalklo} and the fact that $\sgn \xi = \sgn f_{\lambda}(t,x,\xi)$ to conclude that
$$
\begin{aligned}
&\int_{\R^{d+1}_+} \langle |U(\xi)|,\nu_{(t,x)}(\xi)\rangle \varphi(t,x) \dx \dt=\int_{\R^{d+1}_+} (\langle |U(\xi)|,\nu_{(t,x)}(\xi)\rangle -|U(0)|)\varphi(t,x) \dx \dt \\
&=\lim_{\lambda\to \infty}\int_{\R^{d+1}_+} \int_{\R} (U^{\lambda}(\xi))'(\sgn \xi) f_{\lambda}(t,x,\xi) \varphi(t,x) \dxi \dx\dt\\
&=\lim_{\lambda\to \infty}\int_{\R^{d+1}_+} \int_{\R} (U^{\lambda}(\xi))'|f_{\lambda}(t,x,\xi)| \varphi(t,x) \dxi \dx\dt\\
&\le \limsup_{\lambda\to \infty} \int_0^T \|\varphi(t)\|_{L^{\infty}(\R^{d})} \int_{\R^{d+1}}(U^{\lambda}(\xi))'|f_{\lambda}(t,x,\xi)|  \dxi \dx\dt\\
&\le  C(L,T)\|u_0\|_1\int_0^T \|\varphi(t)\|_{L^{\infty}(\R^{d})} \dt,
\end{aligned}
$$
where for the last inequality we used the a~priori bound \eqref{unifflam}. Due to the density argument, we can now choose $\varphi(t,x):=\bbbone_{\{(\tau,\tau+h)\times B_R(0)\}}$ and by letting $R\to \infty$ we get by using the monotone convergence theorem
$$
\begin{aligned}
&\int_{\tau}^{\tau+h}\int_{\R^{d}} \langle |U(\xi)|,\nu_{(t,x)}(\xi)\rangle \dx \dt\le  C(L,T)\|u_0\|_1 h.
\end{aligned}
$$
Consequently, since almost all times $t$ are Lebesgue points, we finally deduce
\begin{equation}
\essup_{t\in(0,T)} \int_{\R^d} \langle |U(\xi)|,\nu_{(t,x)}(\xi)\rangle \dx \le C(L,T)\|u_0\|_1.\label{apL1}
\end{equation}

At this level, we are prepared to use the first stability inequality \eqref{resultin2}. Since \eqref{resultin2} deals only with positive parts, we can first change the role of $U^1$ and $U^2$ in \eqref{resultin2} to obtain
\begin{equation}
\begin{aligned}
&-\int_0^T \int_{\R^d} \langle(U^2(k)-U^1(\lambda))_+,\nu^{1}_{(t,x)}(\lambda)\otimes\nu^{2}_{(t,x)}(k)\rangle \partial_t\varphi(t,x) \dx \dt\\
&\quad- \int_0^T \int_{\R^d}  \langle (\bA^2(k)-\bA^1(\lambda))\bbbone_{\{\lambda \le \tilde{U}(k)\}},\nu^{1}_{(t,x)}(\lambda)\otimes\nu^{2}_{(t,x)}(k)\rangle \cdot \nabla_x \varphi (t,x)\dx \dt\\
& \le \int_0^T \int_{\R^d} \langle (a^{2}(k)-a^1(\lambda))\bbbone_{\{\lambda \le \tilde{U}(k)\}},\nu^{1}_{(t,x)}(\lambda)\otimes \nu^2_{(t,x)}(k)\rangle \varphi(t,x) \dx\dt\\
&\quad + \int_{\R^d} (U^2(v_0^2(x))-U^1(v_0^1(x)))_+\varphi(0,x) \dx.
\label{resultin2b}
\end{aligned}
\end{equation}
Hence summing \eqref{resultin2} and \eqref{resultin2b} we obtain
\begin{equation}
\begin{aligned}
&-\int_0^T \int_{\R^d} \langle|U^2(k)-U^1(\lambda)|,\nu^{1}_{(t,x)}(\lambda)\otimes\nu^{2}_{(t,x)}(k)\rangle \partial_t\varphi(t,x) \dx \dt\\
&\quad- \int_0^T \int_{\R^d}  \langle (\bA^2(k)-\bA^1(\lambda))\sgn (\tilde{U}(k)-\lambda),\nu^{1}_{(t,x)}(\lambda)\otimes\nu^{2}_{(t,x)}(k)\rangle \cdot \nabla_x \varphi (t,x)\dx \dt\\
& \le \int_0^T \int_{\R^d} \langle (a^{2}(k)-a^1(\lambda))\sgn (\tilde{U}(k)-\lambda),\nu^{1}_{(t,x)}(\lambda)\otimes \nu^2_{(t,x)}(k)\rangle \varphi(t,x) \dx\dt\\
&\quad + \int_{\R^d} |U^2(v_0^2(x))-U^1(v_0^1(x))|\varphi(0,x) \dx.
\label{resultin2c}
\end{aligned}
\end{equation}
Now, we closely follow the method developed in \cite{Sz89a} (see also \cite{BuGwMaSw2011}) but in addition we have to deal with the term $a$. Thus using \eqref{resultin2c} with $\bA^1=\bA^2:=\bA$, $U^1=U^2:=U$, $v_0^1=v_0^2:=v_0$, $a^1=a^2:=a+a_{\ell,m}$, $\tilde{U}(s):=s$ and $\nu^1=\nu^2:=\nu$, we see that all assumptions of Theorem~\ref{KatoT} are satisfied and we can use \eqref{resultin2c} to conclude that with this special choice we have for all nonnegative  $\varphi\in \mathcal{D}(-\infty,T; \mathcal{D}(\R^d))$
\begin{equation}
\begin{aligned}
&-\int_0^T \int_{\R^d} \langle|U(k)-U(\lambda)|,\nu_{(t,x)}(\lambda)\otimes\nu_{(t,x)}(k)\rangle \partial_t\varphi(t,x) \dx \dt\\
&\quad- \int_0^T \int_{\R^d}  \langle (\bA(k)-\bA(\lambda))\sgn (k-\lambda),\nu_{(t,x)}(\lambda)\otimes\nu_{(t,x)}(k)\rangle \cdot \nabla_x \varphi (t,x)\dx \dt\\
& \le \int_0^T \int_{\R^d} \langle (a(k)-a(\lambda))\sgn (k-\lambda),\nu_{(t,x)}(\lambda)\otimes \nu_{(t,x)}(k)\rangle \varphi(t,x) \dx\dt\\
&\quad + \int_0^T \int_{\R^d} \langle (a_{\ell,m}(k)-a_{\ell,m}(\lambda))\sgn (k-\lambda),\nu_{(t,x)}(\lambda)\otimes \nu_{(t,x)}(k)\rangle \varphi(t,x) \dx\dt.
\label{resultin2d}
\end{aligned}
\end{equation}
Next, using the assumption \eqref{Ag_mon} on $a$ (on side Lipschitz continuity) and the fact that $a_{\ell,m}$ is decreasing, we obtain
\begin{equation}
\begin{aligned}
&-\int_0^T \int_{\R^d} \langle|U(k)-U(\lambda)|,\nu_{(t,x)}(\lambda)\otimes\nu_{(t,x)}(k)\rangle \partial_t\varphi(t,x) \dx \dt\\
&\quad + \int_0^T \int_{\R^d} \langle |a_{\ell,m}(k)-a_{\ell,m}(\lambda)|,\nu_{(t,x)}(\lambda)\otimes \nu_{(t,x)}(k)\rangle \varphi(t,x) \dx\dt\\
& \le L\int_0^T \int_{\R^d} \langle |U(k)-U(\lambda)|,\nu_{(t,x)}(\lambda)\otimes \nu_{(t,x)}(k)\rangle \varphi(t,x) \dx\dt\\
&\quad + \int_0^T \int_{\R^d}  \langle |\bA(k)-\bA(\lambda)|,\nu_{(t,x)}(\lambda)\otimes\nu_{(t,x)}(k)\rangle |\nabla_x \varphi (t,x)|\dx \dt.
\label{resultin2e}
\end{aligned}
\end{equation}
Finally, we choose a special function $\varphi$. Note that due to the weak star density we can consider Lipschitz functions compactly supported in $(-\infty,T)\times \R^d$. Hence, for arbitrary $R>0$ we find smooth nonnegative $\psi_R\in \mathcal{D}(\R^d)$ such that $\psi_R(x)=1$ if $|x|\le R$, $\psi_R(x)=0$ if $|x|>2R$ and $R|\nabla \psi(x)|\le C(d)$. Then for arbitrary $\tau \in (0,T)$ we set $\varphi(t,x):=(\tau - t)_+ \psi(x)$ in \eqref{resultin2e}. Doing so and using the assumption on $\bA$ and the fact that $\nu$ is a probability measure, we obtain
\begin{equation*}
\begin{aligned}
&\int_0^{\tau} \int_{\R^d} \langle|U(k)-U(\lambda)|,\nu_{(t,x)}(\lambda)\otimes\nu_{(t,x)}(k)\rangle \psi_R(x) \dx \dt\\
&\quad + \int_0^{\tau} \int_{\R^d} \langle |a_{\ell,m}(k)-a_{\ell,m}(\lambda)|,\nu_{(t,x)}(\lambda)\otimes \nu_{(t,x)}(k)\rangle (\tau-t)_+\psi_R(x) \dx\dt\\
& \le L\int_0^{\tau} \int_{\R^d} \langle |U(k)-U(\lambda)|,\nu_{(t,x)}(\lambda)\otimes \nu_{(t,x)}(k)\rangle (\tau-t)_+\psi_R(x) \dx\dt\\
& + \frac{C(d)T}{R}\int_0^T \int_{B_{2R}\setminus B_{R}}  \!\!\!\!\!\!\!\!\langle |\bA(k)-\bA(0)|+|\bA(\lambda)-\bA(0)|,\nu_{(t,x)}(\lambda)\otimes\nu_{(t,x)}(k)\rangle \dx \dt\\
& = L\int_0^{\tau} \int_{\R^d} \langle |U(k)-U(\lambda)|,\nu_{(t,x)}(\lambda)\otimes \nu_{(t,x)}(k)\rangle (\tau-t)_+\psi_R(x) \dx\dt\\
& + \frac{2C(d)T}{R}\int_0^T \int_{B_{2R}\setminus B_{R}}  \langle |\bA(k)-\bA(0)|,\nu_{(t,x)}(k)\rangle \dx \dt\\
& =: L\int_0^{\tau} \int_{\R^d} \langle |U(k)-U(\lambda)|,\nu_{(t,x)}(\lambda)\otimes \nu_{(t,x)}(k)\rangle (\tau-t)_+\psi_R(x) \dx\dt + I(R).
\end{aligned}
\end{equation*}
Noticing that
$$
\begin{aligned}
&\int_0^{\tau} \int_{\R^d} \langle|U(k)-U(\lambda)|,\nu_{(t,x)}(\lambda)\otimes\nu_{(t,x)}(k)\rangle \psi_R(x) \dx \dt\\
&\quad =\frac{d}{d\tau}\int_0^{\tau} \int_{\R^d} \langle|U(k)-U(\lambda)|,\nu_{(t,x)}(\lambda)\otimes\nu_{(t,x)}(k)\rangle \psi_R(x)(\tau-t) \dx \dt
\end{aligned}
$$
and multiplying the above inequality by $e^{-L\tau}$ we see that it reduces to
\begin{equation*}
\begin{aligned}
&\frac{d}{d\tau}\left(e^{-L\tau}\int_0^{\tau} \int_{\R^d} \langle|U(k)-U(\lambda)|,\nu_{(t,x)}(\lambda)\otimes\nu_{(t,x)}(k)\rangle \psi_R(x)(\tau-t) \dx \dt\right)\\
&\quad + e^{-L\tau}\int_0^{T} \int_{\R^d} \langle |a_{\ell,m}(k)-a_{\ell,m}(\lambda)|,\nu_{(t,x)}(\lambda)\otimes \nu_{(t,x)}(k)\rangle (\tau-t)_+\psi_R(x) \dx\dt\\
& \le I(R).
\end{aligned}
\end{equation*}
Integrating the result over $\tau\in (0,T)$ we finally get that
\begin{equation*}
\begin{aligned}
&\int_0^{T} \int_{\R^d} \langle|U(k)-U(\lambda)|,\nu_{(t,x)}(\lambda)\otimes\nu_{(t,x)}(k)\rangle \psi_R(x)(T-t) \dx \dt\\
& +\int_0^T\int_0^{T} \int_{\R^d} \langle |a_{\ell,m}(k)-a_{\ell,m}(\lambda)|,\nu_{(t,x)}(\lambda)\otimes \nu_{(t,x)}(k)\rangle (\tau-t)_+\psi_R(x) \dx\dt\, \mathrm{d}\tau\\
& \le I(R)C(L,T).
\end{aligned}
\end{equation*}
Finally, we let $R\to \infty$ and consequently $\psi_R\nearrow 1$. Using the monotone convergence theorem, we therefore obtain that
\begin{equation}\label{almf}
\begin{aligned}
&\int_0^{T} \int_{\R^d} \langle|U(k)-U(\lambda)|,\nu_{(t,x)}(\lambda)\otimes\nu_{(t,x)}(k)\rangle(T-t) \dx \dt\\
& +\int_0^T\int_0^{T} \int_{\R^d} \langle |a_{\ell,m}(k)-a_{\ell,m}(\lambda)|,\nu_{(t,x)}(\lambda)\otimes \nu_{(t,x)}(k)\rangle (\tau-t)_+ \dx\dt\, \mathrm{d}\tau\\
& \le C(L,T) \limsup_{R\to \infty}I(R).
\end{aligned}
\end{equation}
We need to estimate the last term. Using the definition of $I_R$, the assumption \eqref{AF_mon},  the Jensen and the H\"{o}lder inequality, we observe that
$$
\begin{aligned}
I(R)&=\frac{2C(d)T}{R}\int_0^T \int_{B_{2R}\setminus B_{R}}  \langle |\bA(k)-\bA(0)|,\nu_{(t,x)}(k)\rangle \dx \dt\\
&\le \frac{2C(d)TL}{R}\int_0^T \int_{B_{2R}\setminus B_{R}}  \langle |U(k)|+|U(k)|^{\alpha},\nu_{(t,x)}(k)\rangle \dx \dt\\
&\le \frac{2C(d)TL}{R}\int_0^T \int_{B_{2R}\setminus B_{R}}  \langle |U(k)|,\nu_{(t,x)}(k)\rangle \dx \dt\\
&+ \frac{2C(d)TL}{R}\int_0^T |B_{2R}\setminus B_{R}|^{1-\alpha}\left(\int_{B_{2R}\setminus B_{R}}  |\langle |U(k)|^{\alpha},\nu_{(t,x)}(k)\rangle|^{\frac{1}{\alpha}} \dx\right)^{\alpha} \dt\\
&\le \frac{2C(d)TL}{R}\int_0^T \int_{\R^d}  \langle |U(k)|,\nu_{(t,x)}(k)\rangle \dx \dt\\
&+ \frac{2C(d)T^{2-\alpha}L}{R^{d(\alpha-\frac{1}{d'})}}\left(\int_0^T \int_{\R^d\setminus B_{R}}  \langle |U(k)|,\nu_{(t,x)}(k)\rangle \dx \dt\right)^{\alpha}.
\end{aligned}
$$
Due to the  a~priori estimate \eqref{apL1}, we have that
$$
\lim_{R\to \infty}\int_0^T \int_{\R^d\setminus B_{R}}  \langle |U(k)|,\nu_{(t,x)}(k)\rangle \dx \dt=0,
$$
and consequently, using the assumption that $\alpha \ge \frac{1}{d'}$ we obtain from the above computation that
$$
I(R)\to 0 \textrm{ as } R\to \infty.
$$
Going back to \eqref{almf}, it implies
\begin{equation*}
\begin{aligned}
&\int_0^{T} \int_{\R^d} \langle|U(k)-U(\lambda)|,\nu_{(t,x)}(\lambda)\otimes\nu_{(t,x)}(k)\rangle(T-t) \dx \dt\\
& +\int_0^T\int_0^{T} \int_{\R^d} \langle |a_{\ell,m}(k)-a_{\ell,m}(\lambda)|,\nu_{(t,x)}(\lambda)\otimes \nu_{(t,x)}(k)\rangle (\tau-t)_+ \dx\dt\, \mathrm{d}\tau =0,
\end{aligned}
\end{equation*}
which implies that for almost all $(t,x)\in (0,T)\times \R^d$ there holds
\begin{equation}
\begin{split}
\langle |a_{\ell,m}(k)-a_{\ell,m}(\lambda)|,\nu_{(t,x)}(\lambda)\otimes \nu_{(t,x)}(k)\rangle &=0,\\ \langle|U(k)-U(\lambda)|,\nu_{(t,x)}(\lambda)\otimes\nu_{(t,x)}(k)\rangle&=0. \label{Szep1}
\end{split}
\end{equation}
Finally, we show that \eqref{Szep1} implies that $\nu_{(t,x)}=\delta_{v(t,x)}$. Indeed, assume that $\lambda_1\neq \lambda_2$ belong to the support of $\nu_{(t,x)}$. Then due to the continuity of $a_{\ell,m}$ and its strict monotonicity, it follows that there exists a constant $C$ such that for all $\lambda$ and $k$ fulfilling
\begin{equation}\label{spi1}
|k-\lambda_2|+|\lambda - \lambda_1|\le \frac{|\lambda_1-\lambda_2|}{4},
\end{equation}
we have
\begin{equation}\label{spi2}
C|a_{\ell,m}(k)-a_{\ell,m}(\lambda)|\ge 1.
\end{equation}
Finally, since $\lambda_{1,2}$ are assumed to be in the support of $\nu_{(t,x)}$, we can find a nonnegative $\psi_1\in \mathcal{C}_0(B_{\frac{|\lambda_1-\lambda_2|}{8}}(\lambda_1))$ and $\psi_2\in \mathcal{C}_0(B_{\frac{|\lambda_1-\lambda_2|}{8}}(\lambda_2))$ such that $\|\psi_1\|_{\infty}\le 1$ and $\|\psi_2\|_{\infty}\le 1$ fulfilling
$$
0< \langle \psi_1, \nu_{(t,x)}\rangle, \qquad 0< \langle \psi_2, \nu_{(t,x)}\rangle.
$$
Therefore, using \eqref{spi1}--\eqref{spi2}, we can deduce
$$
\begin{aligned}
0&<\int_{\R^2} \psi_1(\lambda)\psi_2(k) \,\mathrm{d} \nu_{(t,x)}(\lambda)\,\mathrm{d} \nu_{(t,x)}(k)=\langle \psi_1(\lambda)\psi_2(k),  \nu_{(t,x)}(\lambda) \otimes \nu_{(t,x)}(k) \rangle \\
&\le C\langle \psi_1(\lambda)\psi_2(k)|a_{\ell,m}(k)-a_{\ell,m}(\lambda)|,  \nu_{(t,x)}(\lambda) \otimes \nu_{(t,x)}(k) \rangle \\
&\le C\langle |a_{\ell,m}(k)-a_{\ell,m}(\lambda)|,  \nu_{(t,x)}(\lambda) \otimes \nu_{(t,x)}(k) \rangle=0,
\end{aligned}
$$
which is a contradiction. Therefore, $\nu_{(t,x)}$ is supported in a single point and since it is also the Young measure corresponding to $v$  there holds
$$
\nu_{(t,x)}=\delta_{v(t,x)}
$$
and we see that $v$ or respectively $u:=U(v)$ is an entropy weak solution.

\subsection{Limit $\ell,m \to \infty$}
In the previous step we used the auxiliary function $a_{\ell,m}$ to identify the limiting Young measure. Now we want to remove this function, or more precisely we want to let $a_{\ell,m}\to 0$. To do so, we follow \cite{GwSwWiZi2014} and introduce for any $\ell,m\in \mathbb{N}$ the function
$$
a_{\ell,m}(s):=\frac{1}{\ell}\arctan(|s|)\bbbone_{\{s\le 0\}}-\frac{1}{m}\arctan(s)\bbbone_{\{s\ge 0\}}.
$$
Note that for any $\ell,m$ the function $a_{\ell,m}$ is bounded strictly decreasing and we can therefore use the previous subsection to get the existence of entropy weak solution  $v^{\ell,m}\in L^{\infty}(0,T; L^{\infty}(\R^d))$ (in sense of Definition~\ref{D1} for $(\bA,U,a+a_{\ell,m},u_0)$ fulfilling in addition
\begin{equation}
\begin{split}
\essup_{t\in (0,T)} \|v^{\ell,m}(t)\|_{\infty}&\le C(T,U)\|u_0\|_{\infty},\\
\essup_{t\in (0,T)} \|u^{\ell,m}\|_1=\essup_{t\in (0,T)} \|U(v^{\ell,m})\|_1 &\le C(T,U)\|u_0\|_{1}.\label{lmapr}
\end{split}
\end{equation}
Moreover, this choice of $a_{\ell,m}$ implies that
$$
\begin{aligned}
a_{\ell,m} &\le a_{\ell,n} &&\textrm{for all } \ell,m,n\in \mathbb{N} \textrm{ such that }n\ge m,\\
a_{\ell,m} &\le a_{k,m} &&\textrm{for all } \ell,m,k\in \mathbb{N} \textrm{ such that }\ell\ge k.
\end{aligned}
$$
Let us now consider two weak entropy solution $v^{\ell,m},v^{\ell,n}$ corresponding to $(\bA,U,a)$, the initial data $v_0$ and $a_{\ell,m}$, $a_{\ell,n}$ respectively with arbitrary $\ell$ and $n\ge m$. Defining $a^1:=a+a_{\ell,m}$ and $a^2:=a+a_{\ell,n}$, we see that $a^1\le a^2$. Thus, since $v^{\ell,m}$ and $v^{\ell,n}$ are weak entropy solutions, then $\nu^1_{(t,x)}:=\delta_{v^{\ell,m}(t,x)}$ and $\nu^2_{(t,x)}:=\delta_{v^{\ell,n}(t,x)}$ are corresponding measure valued solutions and we can use \eqref{resultin1} to deduce that
\begin{equation*}
\begin{aligned}
&-\int_0^T\int_{\R^d}\langle(U(\lambda)-U(k))_+,\nu^{1}_{(t,x)}(\lambda)\otimes\nu^{2}_{(t,x)}(k)\rangle  \partial_t \varphi(t,x) \dx \dt \\
&\quad -\int_0^T\int_{\R^d}\langle (\bA(\lambda)-\bA(k))\bbbone_{\{\lambda \ge k\}},\nu^{1}_{(t,x)}(\lambda)\otimes\nu^{2}_{(t,x)}(k)\rangle \cdot \nabla_x \varphi(t,x)  \dx \dt\\
& \le \int_0^T\int_{\R^d}\langle ((a^2(\lambda)-a^{2}(k))\bbbone_{\{\lambda \ge k\}} ,\nu^{1}_{(t,x)}(\lambda)\otimes \nu^{2}_{(t,x)}(k)\rangle \varphi(t,x)  \dx \dt\\
&\quad +\int_{\R^d}\langle(U(v_0^1(x))-U(v_0^2(x)))_+\varphi(0,x)  \dx.
\end{aligned}
\end{equation*}
Due to the fact that $v_0^1=v_0^2=v_0$ and since both measures are Young measures, we have that
\begin{equation*}
\begin{aligned}
&-\int_0^T\int_{\R^d}(u^{\ell,m}-u^{\ell,n})_+  \partial_t \varphi \dx \dt \\
&\quad -\int_0^T\int_{\R^d} (\bA(v^{\ell,m})-\bA(v^{\ell,n}))\bbbone_{\{v^{\ell,m} \ge v^{\ell,n}\}} \cdot \nabla_x \varphi  \dx \dt\\
& \le \int_0^T\int_{\R^d} ((a^2(v^{\ell,m})-a^{2}(v^{\ell,n}))\bbbone_{\{v^{\ell,m} \ge v^{\ell,n}\}} \varphi  \dx \dt.
\end{aligned}
\end{equation*}
Hence, using finally the assumptions \eqref{AF_mon} and \eqref{Ag_mon} and the fact that $a^{\ell,n}$ is decreasing, we get
\begin{equation*}
\begin{aligned}
&-\int_0^T\int_{\R^d}(u^{\ell,m}-u^{\ell,n})_+  \partial_t \varphi \dx \dt \\
&\int_0^T\int_{\R^d}(a^{\ell,n}(v^{\ell,n})-a^{\ell,n}(v^{\ell,m}))_{+} \varphi  \dx \dt\\
& \le L\int_0^T\int_{\R^d}(u^{\ell,m}-u^{\ell,n})_+ \varphi  \dx \dt\\
&\quad +\int_0^T\int_{\R^d}(|u^{\ell,m}|+|u^{\ell,m}|^{\alpha} + |u^{\ell,n}|+|u^{\ell,n}|^{\alpha})|\nabla_x \varphi|  \dx \dt.
\end{aligned}
\end{equation*}
Hence, following the previous subsection we set $\varphi(t,x):=(\tau-t)_+\psi_R$ and letting $R\to \infty$ (note that due to the estimate \eqref{lmapr} the last term vanishes) we observe
\begin{equation}
\begin{aligned}
&\int_0^{\tau}\int_{\R^d}(u^{\ell,m}-u^{\ell,n})_+  \dx \dt \\
&\qquad +\int_0^{\tau}\int_{\R^d}(a^{\ell,n}(v^{\ell,n})-a^{\ell,n}(v^{\ell,m}))_{+} (\tau-t)\dx \dt\\
& \le L\int_0^{\tau}\int_{\R^d}(u^{\ell,m}-u^{\ell,n})_+ (\tau-t)  \dx \dt.\label{resultinlm}
\end{aligned}
\end{equation}
Consequently, following the same procedure as in the previous subsection, we deduce from \eqref{resultinlm} that
$$
(u^{\ell,m}-u^{\ell,n})_+ =(a^{\ell,n}(v^{\ell,n})-a^{\ell,n}(v^{\ell,m}))_{+}=0 \textrm{ a.e. in } (0,T)\times \R^d.
$$
Thus, using the fact that $a^{\ell,n}$ is strictly decreasing, we obtain
\begin{equation}\label{m11}
v^{\ell,m}\le v^{\ell,n} \qquad \textrm{for } n\ge m.
\end{equation}
In the same manner one can show that
\begin{equation}\label{m12}
v^{k,m}\le v^{\ell,m} \qquad \textrm{for }k>\ell.
 \end{equation}

Having such monotonicity relations, we  follow \cite{GwSwWiZi2014}. First, we let $n\to \infty$ and then $\ell \to \infty$. Due to the relation \eqref{m11}, we see that for any  fixed $\ell$, there exists a~measurable $v^{\ell}:\Rdp \to \R$ such that
\begin{equation}
v^{\ell,m}\nearrow v^\ell\qquad \textrm{almost everywhere in }\Rdp.\label{m13}
\end{equation}
Moreover, using the a~priori bound \eqref{lmapr} we have that
\begin{equation}
\begin{split}
\essup_{t\in (0,T)} \|v^{\ell}(t)\|_{\infty}&\le C(T,U)\|u_0\|_{\infty},\\
\essup_{t\in (0,T)} \|u^{\ell}\|_1=\essup_{t\in (0,T)} \|U(v^{\ell})\|_1 &\le C(T,U)\|u_0\|_{1}.\label{lmapr1}
\end{split}
\end{equation}
Similarly, it follows from \eqref{m12} and \eqref{m13} that
\begin{equation}\label{m14}
v^{k}\le v^{\ell} \qquad \textrm{for }k>\ell.
\end{equation}
Thus, we again have a monotone sequence and therefore there exists a~measurable $v:\Rdp \to \R$ such that
\begin{equation}
v^{\ell}\searrow v\qquad \textrm{almost everywhere in }\Rdp.\label{m15}
\end{equation}
Then, it follows from \eqref{lmapr1} that
\begin{equation}
\begin{split}
\essup_{t\in (0,T)} \|v(t)\|_{\infty}&\le C(T,U)\|u_0\|_{\infty},\\
\essup_{t\in (0,T)} \|u\|_1=\essup_{t\in (0,T)} \|U(v)\|_1 &\le C(T,U)\|u_0\|_{1}.\label{lmapr34}
\end{split}
\end{equation}
Due to the point-wise  convergence results \eqref{m13} and \eqref{m15}, it is then easy to complete the limit passage and having the a~priori estimates \eqref{lmapr34}, the function $u$ is the entropy weak solution.

\subsection{Uniqueness}
We focus here on the uniqueness of the solution and its independence on the choice of parametrization. Thus, let $(U^1,\bA^1,a^1)$ and $(U^2,\bA^2,a^2)$ be two equivalent parametrizations fulfilling \eqref{eqpa} and consider two corresponding entropy weak solutions $u^1=U^1(v^1)$ and $u^2=U^2(v^2)$ with some initial data $v_0^1$, $v_0^2$ fulfilling $U^1(v_0^1)=U^2(v^2_0)$ almost everywhere  in  $\R^d$. Since, we have a weak solution, it is evident that $\nu^1_{(t,x)}:=\delta_{v^1(t,x)}$ and $\nu^2_{(t,x)}:=\delta_{v^2(t,x)}$ are two entropy measure valued solutions and therefore we can use the stability inequality \eqref{resultin2} to get (note that the term with initial data vanishes)
\begin{equation*}
\begin{aligned}
&-\int_0^T \int_{\R^d} \langle(U^1(\lambda)-U^2(k))_+,\nu^{1}_{(t,x)}(\lambda)\otimes\nu^{2}_{(t,x)}(k)\rangle \partial_t\varphi(t,x) \dx \dt\\
&\quad- \int_0^T \int_{\R^d}  \langle (\bA^1(\lambda)-\bA^2(k))\bbbone_{\{\lambda \ge \tilde{U}(k)\}},\nu^{1}_{(t,x)}(\lambda)\otimes\nu^{2}_{(t,x)}(k)\rangle \cdot \nabla_x \varphi (t,x)\dx \dt\\
& \le \int_0^T \int_{\R^d} \langle (a^{1}(\lambda)-a^2(k))\bbbone_{\{\lambda \ge \tilde{U}(k)\}},\nu^{1}_{(t,x)}(\lambda)\otimes \nu^2_{(t,x)}(k)\rangle \varphi(t,x) \dx\dt.
\end{aligned}
\end{equation*}
Since we can change the role of both solutions (see the previous subsection) and using the fact that the Young measures are Dirac measures, we get
\begin{equation*}
\begin{aligned}
&-\int_0^T \int_{\R^d} |u^1(t,x)-u^2(t,x)| \partial_t\varphi(t,x) \dx \dt\\
&\;- \int_0^T \int_{\R^d}   (\bA^1(v^1(t,x))-\bA^2(v^2(t,x)))\sgn (v^1(t,x) - \tilde{U}(v^2(t,x))) \cdot \nabla_x \varphi (t,x)\dx \dt\\
& \le \int_0^T \int_{\R^d}  (a^{1}(v^1(t,x))-a^2(v^2(t,x)))\sgn (v^1(t,x) - \tilde{U}(v^2(t,x))) \varphi(t,x) \dx\dt,
\end{aligned}
\end{equation*}
which after using \eqref{eqpa} leads to the estimate
\begin{equation*}
\begin{aligned}
&-\int_0^T \int_{\R^d} |u^1(t,x)-u^2(t,x)| \partial_t\varphi(t,x) \dx \dt\\
&\le \int_0^T \int_{\R^d}   |\bA^1(v^1(t,x))-\bA^2(v^2(t,x))||\nabla_x \varphi (t,x)|\dx \dt\\
& + \int_0^T \int_{\R^d}  (a^{1}(v^1(t,x))-a^1(\tilde{U}(v^2(t,x)))\sgn (v^1(t,x) - \tilde{U}(v^2(t,x))) \varphi(t,x) \dx\dt.
\end{aligned}
\end{equation*}
Using finally the assumption \eqref{strictnm} (in case we do not assume the strict monotonicity relation but only \eqref{Ag_mon} we simply set  $g(s):=0$ in what follows)
\begin{equation}
\begin{aligned}
&-\int_0^T \int_{\R^d} |u^1(t,x)-u^2(t,x)| \partial_t\varphi(t,x) \dx \dt\\
&+\int_0^T \int_{\R^d}  g(v^1(t,x)-\tilde{U}(v^2(t,x))) \varphi(t,x) \dx\dt\\
&\le \int_0^T \int_{\R^d}   |\bA^1(v^1(t,x))-\bA^2(v^2(t,x))||\nabla_x \varphi (t,x)|\dx \dt\\
& + L\int_0^T \int_{\R^d} ( U^1(v^1(t,x))-U^1(\tilde{U}(v^2(t,x)))) \varphi(t,x) \dx\dt\\
&\le \int_0^T \int_{\R^d}   |\bA^1(v^1(t,x))-\bA^2(v^2(t,x))||\nabla_x \varphi (t,x)|\dx \dt\\
& + L\int_0^T \int_{\R^d} |u^1(t,x)-u^2(t,x)| \varphi(t,x) \dx\dt.
\label{resultin2un}
\end{aligned}
\end{equation}
Since, $g$ is nonnegative we can directly follow the procedure from the previous section to conclude that $u^1=u^2$ almost everywhere in $\Rdp$. In addition, it follows from \eqref{resultin2un} that $g(v^1-\tilde{U}(v^2))=0$ almost everywhere and consequently using the assumption on the strict positivity of $g$ everywhere except zero, we deduce that $v^1=\tilde{U}(v^2)$ almost everywhere in $\Rdp$.

\subsection{Continuity with respect to $t$}

In this final subsection, we show that the unique entropy weak solution $u$ belongs to the space $\mathcal{C}([0,T]; L^1(\R^d))$. Since $u$ is also a bounded distributional solution then we directly obtain that for any $p\in [1,\infty)$
\begin{equation}\label{weakcont}
u\in \mathcal{C}_{weak}(0,T; L^p_{loc}(\R^d)),
\end{equation}
i.e., $u$ is weakly continuous and it makes sense to define it for all times $t\in [0,T]$. Our goal is to strengthen \eqref{weakcont} and to show that it is in fact strongly continuous. For this purpose, we again use \eqref{resultin2}. Indeed, setting $\nu^1_{(t,x)}:=\delta_{v(t+h,x)}$ and $\nu^2_{(t,x)}:=\delta_{v(t,x)}$, which are two entropy measure valued solutions in $(0,T-h)\times \R^d$ and $(0,T)\times \R^d$. Therefore, we can in the very similar manner as before deduce that for all nonnegative $\varphi \in \mathcal{D}((0,T-h)\times \R^d)$ there holds
\begin{equation}
\begin{aligned}
&-\int_0^{T-h} \int_{\R^d} |u(t+h,x)-u(t,x)| \partial_t\varphi(t,x) \dx \dt\\
&\le \int_0^{T-h} \int_{\R^d}   |\bA(v(t+h,x))-\bA(v(t,x))||\nabla_x \varphi (t,x)|\dx \dt\\
& + L\int_0^{T-h} \int_{\R^d} |u(t+h,x)-u(t,x)| \varphi(t,x) \dx\dt.
\label{resultin2con}
\end{aligned}
\end{equation}
The reason why we have to consider the compactly supported functions in $(0,T-h)$ is that up to now we do not know that the value $u(h)$ is attained continuously. Nevertheless, by the density argument, we can set $\varphi(t,x):=\psi(x)\phi(t)$ in \eqref{resultin2con}, where $\psi \in \mathcal{D}(\R^d)$ is a nonnegative function and $\phi(t)$ is chosen such that
$$
\phi(t):=\left\{\begin{aligned}
&\frac{t}{t_1} &&\textrm{for } t\in [0,t_1],\\
&1 &&\textrm{for } t\in (t_1,t_2),\\
&1-\frac{t-t_2}{\tau} &&\textrm{for } t\in[t_2,t_2+\tau],\\
&0 &&\textrm{for } t\in (t_2+\tau,T].
\end{aligned} \right.
$$
Doing so, we obtain
\begin{equation}
\begin{aligned}
&\frac{1}{\tau}\int_{t_2}^{t_2+\tau}\int_{\R^d} |u(t+h,x)-u(t,x)| \psi(x) \dx \dt\\
&\le \int_0^{T-h} \int_{\R^d}   |\bA(v(t+h,x))-\bA(v(t,x))||\nabla_x \psi (x)|\dx \dt\\
& + L\int_0^{T-h} \int_{\R^d} |u(t+h,x)-u(t,x)| \psi(x) \dx\dt\\
&+\frac{1}{t_1}\int_{0}^{t_1}\int_{\R^d} |u(t+h,x)-u(t,x)| \psi(x) \dx \dt.
\label{resultin2con12}
\end{aligned}
\end{equation}
Using the weak continuity \eqref{weakcont} and weak lower semicontinuity we can let $\tau\to 0_+$ in \eqref{resultin2con12} to obtain
\begin{equation}
\begin{aligned}
&\int_{\R^d} |u(t_2+h,x)-u(t_2,x)| \psi(x) \dx \\
&\le \int_0^{T-h} \int_{\R^d}   |\bA(v(t+h,x))-\bA(v(t,x))||\nabla_x \psi (x)|\dx \dt\\
& + L\int_0^{T-h} \int_{\R^d} |u(t+h,x)-u(t,x)| \psi(x) \dx\dt\\
&+\frac{1}{t_1}\int_{0}^{t_1}\int_{\R^d} |u(t+h,x)-u(t,x)| \psi(x) \dx \dt.
\label{resultin2con13}
\end{aligned}
\end{equation}
Letting now $\psi \nearrow 1$ (note here that the first term on the right hand side vanishes, while for the limiting procedure we can use the monotone convergence theorem to get
\begin{equation}
\begin{aligned}
&\|u(t_2+h)-u(t_2)\|_1\le C(L, t_1)\int_0^{T-h} \|u(t+h)-u(t)\|_1\dt.
\label{contend}
\end{aligned}
\end{equation}
Thus, we see that $u\in \mathcal{C}_{loc}(0,T; L^1(\R^d))$. To prove the continuity up to the initial time $t=0$, we show that
\begin{equation}\label{stronginit}
\lim_{t\to 0_+} \|u(t)-u_0\|_{1}=0.
\end{equation}
Repeating almost step by step the proof of \eqref{indatameasure}, we can deduce that for any compact set $K\subset \R^d$ there holds
\begin{equation}\label{strongi1}
\lim_{t\to 0_+} \|u(t)-u_0\|_{L^1(K)}.
\end{equation}
Thus, it remains to show a uniform decay of the solution at infinity. Since we already have continuity of the solution with respect to time, we proceed more formally. Using the entropy inequality with the entropy $E(s):=|s|$ we have that
$$
\partial_t |u| +\diver \left ((\bA(v)-\bA(0))\sgn v\right) \le a(v)\sgn v \le L|u|,
$$
where for the second inequality, we used \eqref{Ag_mon}. Testing this equation by a nonnegative $\psi \in \mathcal{D}(\R^d)$, we have that
$$
\begin{aligned}
&\frac{d}{dt} \int_{\R^d} |u(t,x)|\psi(x)\dx \le L\int_{\R^d} |u(t,x)|\psi(x)\dx +\int_{\R^d}|\bA(v(t,x))-\bA(0)||\nabla_x \psi(x)|\dx\\
&\le L\int_{\R^d} |u(t,x)|\psi(x)\dx +\int_{\R^d}(|u(t,x)|+ |u(t,x)|^{\alpha})|\nabla_x \psi(x)|\dx,
\end{aligned}
$$
where for the second inequality we used the assumption \eqref{AF_mon}. Consequently, using the Gronwall inequality we deduce
$$
\begin{aligned}
\int_{\R^d} |u(t,x)|\psi(x)\dx \le &e^{Lt}\int_{\R^d} |u_0(x)|\psi(x)\dx \\
&+e^{LT}\int_0^T\int_{\R^d}(|u(s,x)|+ |u(s,x)|^{\alpha})|\nabla_x \psi(x)|\dx\ds.
\end{aligned}
$$
Finally, we can let $\psi\nearrow \psi_R$, where $\psi_R=1$ in $\R^d\setminus B_{2R}(0)$ and $\psi_R=0$ in $B_{R}(0)$ such that $|\nabla \psi_R|\le C R^{-1}$ to conclude
\begin{equation}\label{konec}
\begin{aligned}
&\int_{\R^d\setminus B_{2R}(0)} |u(t,x)|\dx \le e^{Lt}\int_{\R^d\setminus B_{R}(0)} |u_0(x)|\dx \\
&\qquad +Ce^{LT}(R^{-1} + R^{\frac{d}{d'}-d\alpha}\left(\int_0^T\int_{\R^d \setminus B_{R}(0)}|u(s,x)|\dx\ds \right)^{\frac{1}{\alpha}}.
\end{aligned}
\end{equation}
Thus, finally, to obtain \eqref{stronginit}, we have
\begin{equation*}
\begin{split}
&\lim_{t\to 0_+} \|u(t)-u_0\|_{1}=\lim_{t\to 0_+} (\|u(t)-u_0\|_{L^1(B_{R}(0))}+\|u(t)-u_0\|_{L^1(\R^d\setminus B_{R}(0))}\\
& \overset{\eqref{strongi1}}=\lim_{t\to 0_+}\|u(t)-u_0\|_{L^1(\R^d\setminus B_{R}(0))} \\
&\overset{\eqref{konec}}\le C(L,T)\int_{\R^d\setminus B_{R}(0)} |u_0(x)|\dx +C(L,T)\left(\int_0^T\int_{\R^d \setminus B_{R}(0)}|u(s,x)|\dx\ds \right)^{\frac{1}{\alpha}}\\
&\overset{R\to \infty}\to 0,
\end{split}
\end{equation*}
which finishes the proof.

\bibliographystyle{abbrv}
\bibliography{hyper}

\begin{thebibliography}{10}

\bibitem{AudussePerthame}
E.~Audusse and B.~Perthame.
\newblock Uniqueness for scalar conservation laws with discontinuous flux via
  adapted entropies.
\newblock {\em Proc. Roy. Soc. Edinburgh Sect. A}, 135(2):253--265, 2005.

\bibitem{Bo2002}
R.~Botchorishvili.
\newblock Implicit kinetic schemes for scalar conservation laws.
\newblock {\em Numer. Methods Partial Differential Equations}, 18(1):26--43,
  2002.

\bibitem{BuGwGw13}
M.~Bul{\'{\i}}{\v{c}}ek, P.~Gwiazda, and A.~{\'S}wierczewska-Gwiazda.
\newblock Multi-dimensional scalar conservation laws with fluxes discontinuous
  in the unknown and the spatial variable.
\newblock {\em Math. Models Methods Appl. Sci.}, 23(3):407--439, 2013.

\bibitem{BuGwMaSw2011}
M.~Bul\'{\i}\v{c}ek, P.~Gwiazda, J.~M\'{a}lek, and A.~\'{S}wierczewska Gwiazda.
\newblock On scalar hyperbolic conservation laws with a discontinuous flux.
\newblock {\em Math. Models Methods Appl. Sci.}, 21(1):89--113, 2011.

\bibitem{Carrillo2003}
J.~Carrillo.
\newblock Conservation laws with discontinuous flux functions and boundary
  condition.
\newblock {\em J. Evol. Equ.}, 3(2):283--301, 2003.

\bibitem{DiPerna}
R.~J. DiPerna.
\newblock Measure-valued solutions to conservation laws.
\newblock {\em Arch. Rational Mech. Anal.}, 88(3):223--270, 1985.

\bibitem{Ev2010}
L.~C. Evans.
\newblock {\em Partial differential equations}, volume~19 of {\em Graduate
  Studies in Mathematics}.
\newblock American Mathematical Society, Providence, RI, second edition, 2010.

\bibitem{GwSw2005}
P.~Gwiazda and A.~{{\'S}}wierczewska.
\newblock Multivalued equations for granular avalanches.
\newblock {\em Nonlinear Anal.}, 62(5):895--912, 2005.

\bibitem{GwSwWiZi2014}
P.~Gwiazda, A.~{{\'S}}wierczewska Gwiazda, P.~Wittbold, and A.~Zimmermann.
\newblock Multi-dimensional scalar balance laws with discontinuous flux.
\newblock {\em J. Funct. Anal.}, 267(8):2846--2883, 2014.

\bibitem{Kr70}
S.~N. Kru{\v{z}}kov.
\newblock First order quasilinear equations with several independent variables.
\newblock {\em Mat. Sb. (N.S.)}, 81 (123):228--255, 1970.

\bibitem{MaPe2003}
C.~Makridakis and B.~Perthame.
\newblock Sharp {CFL}, discrete kinetic formulation, and entropic schemes for
  scalar conservation laws.
\newblock {\em SIAM J. Numer. Anal.}, 41(3):1032--1051, 2003.

\bibitem{Panov}
E.~Y. Panov.
\newblock On existence and uniqueness of entropy solutions to the {C}auchy
  problem for a conservation law with discontinuous flux.
\newblock {\em J. Hyperbolic Differ. Equ.}, 6(3):525--548, 2009.

\bibitem{Perthame02}
B.~Perthame.
\newblock {\em Kinetic formulation of conservation laws}, volume~21 of {\em
  Oxford Lecture Series in Mathematics and its Applications}.
\newblock Oxford University Press, Oxford, 2002.

\bibitem{PT91}
B.~Perthame and E.~Tadmor.
\newblock {A kinetic equation with kinetic entropy functions for scalar
  conservation laws.}
\newblock {\em Commun. Math. Phys.}, 136(3):501--517, 1991.

\bibitem{Sz89a}
A.~Szepessy.
\newblock An existence result for scalar conservation laws using measure valued
  solutions.
\newblock {\em Comm. Partial Differential Equations}, 14(10):1329--1350, 1989.

\end{thebibliography}

\end{document}